\documentclass[11pt]{amsart}

\usepackage{a4wide}
\usepackage{amsmath}
\usepackage{amsfonts}
\usepackage{amssymb}
\usepackage{amsthm}
\usepackage{color}

\numberwithin{equation}{section}

\newtheorem{theorem}{Theorem}[section]
\newtheorem{lemma}[theorem]{Lemma}
\newtheorem{corollary}[theorem]{Corollary}
\newtheorem{remark}[theorem]{Remark}
\newtheorem{proposition}[theorem]{Proposition}

\renewcommand{\d}{\mathrm{d}}
\newcommand{\dd}{\,\mathrm{d}}
\newcommand{\var}{\text{-}\mathrm{var}}
\newcommand{\hol}{\text{-}\mathrm{H}\textnormal{\"o}\mathrm{l}}
\newcommand{\om}{\text{-}\omega}
\newcommand{\R}{\mathbb{R}}
\newcommand{\pr}{\prime}
\newcommand{\Lip}{\mathrm{Lip}}
\newcommand{\X}{{\bf X}}

\renewcommand{\d}{\mathrm{d}}

\begin{document}

\title{Rough path metrics on a Besov--Nikolskii type scale}

\author[Friz]{Peter K. Friz}
\address{Peter K. Friz, Technische Universit\"at Berlin and Weierstrass Institute Berlin, Germany}

\author[Pr\"omel]{David J. Pr\"omel}\date{\today}
\address{David J. Pr\"omel, Eidgen\"ossische Technische Hochschule Z\"urich, Switzerland}

\begin{abstract}
  It is known, since the seminal work [T. Lyons, Differential equations driven by rough signals, Rev. Mat. Iberoamericana, 14 (1998)], that the solution map associated to a controlled differential equation is locally Lipschitz continuous in $q$-variation resp. $1/q$-H\"{o}lder type metrics on the space of rough paths, for any regularity $1/q \in (0,1]$.
  
  We extend this to a new class of Besov-Nikolskii-type metrics, with arbitrary regularity $1/q\in (0,1]$ and integrability $p\in [ q,\infty ]$, where the case $p\in \{ q,\infty \} $ corresponds to the known cases. Interestingly, the result is obtained as consequence of known $q$-variation rough path estimates.
\end{abstract}

\maketitle



\section{Introduction}

We are interested in controlled differential equations of the type 
\begin{equation}\label{eq:cODEintro}
  \d Y_{t}=V(Y_{t})\,\mathrm{d}X_{t},\quad t\in \lbrack 0,T],  
\end{equation}
where $X=(X_{t})$ is a suitable ($n$-dimensional) driving signal, $Y=(Y_{t})$ is the ($m$-dimensional) output signal and $V=(V_{1},\dots ,V_{n})$ are vector fields of suitable regularity. A fundamental question concerns the continuity of the solution map $X\mapsto Y$, strongly dependent on the used metric. 

\medskip

A decisive answer is given by rough path theory, which identifies a cascade of good metrics, determined by some \textit{regularity parameter} $\delta \equiv 1/q\in (0,1]$, and essentially given by $q$-variation resp. $\delta $-H\"{o}lder type metrics. As long as the driving signal $X$ possesses sufficient regularity, say $X$ is a continuous path of finite $q$-variation for $q\in [1,2)$ (in symbols $X\in C^{q\var}([0,T];\mathbb{R}^n)$), Lyons~\cite{Lyons1994} showed that the solution map $X\mapsto Y$ associated to equation~\eqref{eq:cODEintro} is a locally Lipschitz continuous map with respect to the $q$-variation topology. However, this strong regularity assumption on $X$ excludes many prominent examples from probability theory as sample paths of stochastic processes like (fractional) Brownian motion, martingales or many Gaussian processes.

In order to restore the continuity of the solution map associated to a controlled differential equation for continuous paths $X$ of finite $q$-variation for arbitrary large $q<\infty$, it is not sufficient anymore to consider a path ``only'' taking in the Euclidean space~$\R^n$, cf.~\cite{Lyons1991,Lyons2007}. Instead, $X$ must be viewed as $\lfloor q\rfloor$-level rough path, which in particular means $X$ takes values in a step-$\lfloor q\rfloor$ free nilpotent group~$G^{\lfloor q\rfloor}(\mathbb{R}^n)$: Let us recall that for $Z\in C^{1\var}(\mathbb{R}^n)$ its $\lfloor q\rfloor$-step signature is given by 
\begin{align*}
  S_{\lfloor q\rfloor}(Z)_{s,t}:=&\bigg (1, \int_{s<u<t}\dd Z_u, \dots, \int_{s<u_1<\dots <u_{\lfloor q\rfloor}<t}\dd Z_{u_1} \otimes \cdots \otimes \d Z_{u_{\lfloor q\rfloor}} \bigg). \\
\end{align*}
The corresponding space of all these lifted paths is 
\begin{equation*}
  G^{\lfloor q\rfloor}(\mathbb{R}^n):= \{S_{\lfloor q\rfloor}(Z)_{0,T} \,:\, Z\in C^{1\var}([0,T];\mathbb{R}^n)\}\subset \bigoplus_{k=0}^{\lfloor q\rfloor} \big(\mathbb{R}^n\big)^{\otimes k},
\end{equation*}
which we equip with the Carnot-Caratheodory metric $d_{cc}$, see Subsection~\ref{subsection:rough path} for more details. While in the case of $q\in [1,2)$ this reduces to a classical path $X\colon [0,T]\rightarrow \mathbb{R}^{n}$, in the case of $q > 2$ this means, intuitively, $X$ is a path enhanced with the information corresponding to the ``iterated integrals'' up to order $\lfloor q\rfloor$. In the context of rough path theory, the solution map $X\mapsto Y$, taking now a $\lfloor q\rfloor$-level rough path $X$ (in symbols $X\in C^{q\var}([0,T],G^{\lfloor q\rfloor}(\mathbb{R}^n))$) as input, is often called \textit{It\^o-Lyons map}.

In most applications, the output is regarded as path, $Y\in C^{q\text{-}\mathrm{var}}([0,T];\mathbb{R}^{m})$, although - depending on the
route one takes - it can be seen as rough path \cite{Lyons1998,Lyons2002,Friz2010} or controlled rough path \cite{Gubinelli2004,Friz2014}. It is a fundamental property of rough path theory that solving differential equations - that is, applying the It\^{o}-Lyons map - entails no loss of regularity: if the driving signal enjoys $\delta$-H\"{o}lder (resp. $q$-variation) regularity, then so does the output signal.

\medskip

Let us explain the basic idea which underlies this work. To this end -- only estimates matter -- take $X$ smooth and rewrite \eqref{eq:cODEintro} in the classical form $\dot{Y}=V(Y) \dot{X}$. Take $L^{p}$-norms on both sides to arrive at 
\begin{equation}\label{eq:W1intro}
  \left\Vert Y\right\Vert _{W^{1,p};\left[ 0,T\right] }\leq \left\Vert V\right\Vert _{\infty }\left\Vert X\right\Vert _{W^{1,p};\left[ 0,T\right] }
\end{equation}
in terms of the semi-norm $\left\Vert X\right\Vert _{W^{1,p};\left[0,T\right]}:= (\int_{0}^{T}|\dot{X}_{t}|^{p}\dd t)^{1/p}$. Here, of course, we have regularity $\delta =1~(\Leftrightarrow q=1$), and the extreme cases $p\in \left\{ 1,\infty \right\} $ $(=\{q,\infty \})$ amount exactly to the variation resp. H\"{o}lder estimates
\begin{align}\label{eq:VarHolIntro}
  \begin{split}
    &\left\Vert Y\right\Vert _{1\text{-var};\left[ 0,T\right] } \leq \left\Vert V \right\Vert_{\infty }\left\Vert X\right\Vert _{1\text{-var};\left[ 0,T \right] }, \\
    &\left\Vert Y\right\Vert _{1\text{-H\"{o}l};\left[ 0,T\right] } \leq \left\Vert V\right\Vert_{\infty }\left\Vert X\right\Vert _{1\text{-H\"{o}l};\left[ 0,T\right]}  
  \end{split}
\end{align}
since indeed $\left\Vert X\right\Vert _{1\text{-var};\left[ 0,T\right]}\approx\left\Vert X\right\Vert _{W^{1,1}\left[ 0,T\right] }$ resp. $\left\Vert X\right\Vert _{1\text{-H\"{o}l;}\left[ 0,T\right] }=\left\Vert X\right\Vert_{W^{1,\infty };\left[ 0,T\right] }$.  Conversely, one may view \eqref{eq:W1intro} as \textit{interpolation} of the estimates~\eqref{eq:VarHolIntro}, by regarding $W^{1,p}$, for any $p\in \lbrack 1,\infty ]$, as interpolation space of $W^{1,1}$ and $W^{1,\infty }$. This discussion suggests moreover that the solution map $X \mapsto Y$ is also continuous in $W^{1,p}$, even locally Lipschitz in the sense
\begin{equation}\label{eq:W1pCont}
  \| Y^1 - Y^2 \|_{W^{1,p};[0,T]} \lesssim  \| X^1 - X^2 \|_{W^{1,p};[0,T]},
\end{equation}
as indeed may be seen by some fairly elementary analysis. (Mind, however, that the solution map $X \mapsto Y$ is highly non-linear so that there is little hope to appeal to some ``general theory of interpolation''.) 

The estimates \eqref{eq:VarHolIntro} and \eqref{eq:W1pCont}, in case $p=1$ and $p=\infty$, are well-known (e.g. \cite{Lyons1998, Lyons2002, Friz2010}) to extend to arbitrarily low regularity $\delta \equiv 1/q\in (0,1]$, provided that, essentially, $\left\Vert\, \cdot\, \right\Vert _{1\text{-var};\left[ 0,T\right] }$ is replaced by $\left\Vert\, \cdot\, \right\Vert _{q\text{-var};\left[ 0,T\right] }$ (with the correct rough path interpretation on the right-hand sides above).

\medskip 

The question arises if the well-studied $q$-variation and $\delta$-H\"{o}lder formulation of rough path theory are not the extreme cases of a more flexible formulation of the theory, that comes - in the spirit of Besov (Nikolskii) spaces - with an additional {\it integrability parameter} $p$ $\in \lbrack q,\infty ]$. (Here, having $q$ as lower bound on $p$ is quite natural in view of known Besov embeddings: in the Besov-scale $(B_{r}^{\delta ,p})$, with additional fine-tuning parameter $r$, one has, always with $\delta = 1/q$, $C^{q\text{-var}}\approx N^{\delta,q}\equiv B_{\infty }^{\delta,q}$, in the form of tight (but strict) inclusions $N^{\delta +\varepsilon ,q}\subset C^{q\var} \subset N^{\delta ,q}$; see Remark \ref{rmk:besovvar}.)

\medskip 

The first contribution of this paper is to given an affirmative answer to the above question, in the generality of arbitrarily low  regularity $\delta >0$. With focus on the interesting case of regularity $\delta <1$, we have, loosely stated,

\begin{theorem}\label{thm:ThmOneIntro}
  Let $\delta \equiv 1/q\in (0,1]$ and $p$ $\in \lbrack q,\infty ]$. Then, for $Lip^{\gamma}$ vector fields~$V$ with $\gamma>q$, the It\^{o}-Lyons map (as defined below in~\eqref{eq:ito lyons map}) is locally Lipschitz continuous from a Besov-Nikolskii-type (rough) path space with regularity/integrability~$\left( \delta ,p\right)$ into a Besov-Nikolskii-type path space of identical regularity/integrability~$\left( \delta ,p\right)$ .
\end{theorem}

Somewhat surprisingly, it is possible to prove this via delicate application of classical $q$-variation estimates\footnote{The use of control functions is pure notational convenience.} in rough path theory; that is, morally, from the case $p=q$. On the other hand, a precise definition of the involved spaces - to make this reasoning possible - is a subtle matter. First, care is necessary for rough paths take values in a non-linear space, the step-$\lfloor q\rfloor$ free nilpotent group~$G^{\lfloor q\rfloor}(\mathbb{R}^n)$ equipped with the Carnot-Caratheodory metric~$d_{cc}$, which is a no standard setting for classical Besov resp. Nikolskii spaces $(B_{r}^{\delta ,p}$ resp. $N^{\delta
,p}:=B_{\infty }^{\delta ,p}$). Another and quite serious difficulty is the lack of super-additivity of Nikolskii norms. Recall that the "control"
\begin{equation*}
  \omega \left( s,t\right) :=\left\Vert X\right\Vert _{p\text{-var;}\left[ s,t\right] }^{p}
\end{equation*}
has the most desirable property of super-additivity, i.e. $\omega \left(s,t\right) +\omega \left( t,u\right) \leq \omega \left( s,u\right) $, a simple fact that is used throughout Lyons' theory. For instance, as a typical consequence 
\begin{equation*}
  \left\Vert X\right\Vert _{p\text{-var};\left[ 0,T\right] }^{p}=\sup_{\mathcal{P}\subset \lbrack 0,T]}\sum_{[u,v]\in \mathcal{P}}\left\Vert X\right\Vert _{p\text{-var};[u,v]}^{p},
\end{equation*}
where the supremum is taking over all partition of the interval $[0,T]$. Several other (rough) path space norms also have this property, as exploited e.g. in \cite{Friz2006}. However, this convenient property fails for the Besov spaces of consideration (unless $\delta =1$) and indeed, in general with strict inequality,  
\begin{equation*}
  \left\Vert X\right\Vert _{N^{\delta ,p}\text{;}\left[ 0,T\right] }^{p}\leq \sup_{\mathcal{P}\subset \lbrack 0,T]}\sum_{[u,v]\in \mathcal{P}}\left\Vert X\right\Vert _{N^{\delta ,p}\text{;}[u,v]}^{p}=:\left\Vert X\right\Vert _{\hat{N}^{\delta ,p}\text{;}\left[ 0,T\right] }^{p}.
\end{equation*}
This leads us to use the \textit{Besov-Nikolskii-type space} $\hat{N}^{\delta ,p}$, defined as those $X$ for which the right-hand side above is finite, as the correct space (in rough path or path space incarnation) to which we refer in Theorem~\ref{thm:ThmOneIntro}, at least in the new regimes $\delta<1,\,p\in (q,\infty )$.  

\medskip

A better understanding of these spaces is compulsory, and this is the second contribution of this paper. For instance, it is reassuring that one has tight inclusions of the form $N^{\delta +\varepsilon,p}\subset \hat{N}^{\delta ,p}\subset N^{\delta ,p}$ (Corollary~\ref{cor:212}). In fact, an exact characterization is possible in terms of \textit{Riesz type variation} spaces, in reference to Riesz~\cite{Riesz1910}, who considered such spaces (although with regularity parameter $\delta =1$). We have

\begin{theorem}\label{thm:ThmTwoIntro}
  Consider $\delta =1/q<1$ and $p\in (q,\infty )$. Then the Besov-Nikolskii-type space $\hat{N}^{\delta ,p}$ coincide with the Riesz type variation spaces $V^{\delta,p}$ and $\tilde{V}^{\delta ,p}$ defined respectively via finiteness of 
  \begin{eqnarray*}
    \left\Vert X\right\Vert_{V^{\delta ,p}}^{p} &:= &\sup_{\mathcal{P}\subset \lbrack 0,T]}\sum_{[u,v]\in \mathcal{P}}\frac{d_{cc}(X_{v},X_{u})^{p}}{|v-u|^{\delta p-1}}, \\ 
    \left\Vert X\right\Vert_{\tilde{V}^{\delta ,p}}^{p} &:= &\sup_{\mathcal{P}\subset \lbrack 0,T]}\sum_{[u,v]\in \mathcal{P}}\frac{\Vert X\Vert _{\frac{1}{\delta }\text{-}\mathrm{var};[u,v]}^{p}}{|v-u|^{\delta p-1}},
  \end{eqnarray*}
  for a rough path $X$ and the Carnot-Caratheodory distance $d_{cc}$. More general, this is also true for arbitrary metric spaces instead of $G^{\lfloor q\rfloor}(\mathbb{R}^n)$.
  
  Moreover, all associated inhomogenous rough path distances are locally Lipschitz equivalent.
\end{theorem}

Let us also note that the above introduced Riesz type variation spaces agree (trivially) with the $q$-variation space in the extreme case of $p=q \equiv 1/\delta$. (In the Besov scale, this usually fails. For instance, we have the strict inclusion $W^{1,1} \subset C^{1\var}$; not every rectifiable path is absolutely continuous.) 

\medskip

We conclude this introduction with some pointers to previous works. The case of regularity $\delta >1/2$, essentially a Young regime, was considered in \cite{Zahle1998,Zahle2001}. Our result can also be regarded as extension of \cite{Promel2016}, which effectively dealt with regularity $\delta =1/q>1/3$ and accordingly integrability $p\geq q=3$. We note that path spaces with \textquotedblleft mixed\textquotedblright\ H\"{o}lder-variation regularity, similar in spirit to the Riesz type spaces (with tilde) also appear as tangent spaces to H\"{o}lder rough path spaces \cite[p.209]{Friz2010}, see also \cite{Aida2016}. Moreover, regularity of Cameron-Martin spaces associated to Gaussian processes with \textquotedblleft H\"{o}lder dominated $\rho$-variation of the covariance\textquotedblright\ (a key condition in Gaussian rough path theory, cf. \cite[Ch.~10]{Friz2014}, \cite{Friz2016}) can be expressed with the help of \textquotedblleft mixed\textquotedblright\ H\"older-variation regularity, see e.g. \cite[p.151]{Friz2014}. 

\medskip 

\noindent \textbf{Organization of the paper:} In Section~\ref{sec:riesz} we define and give various characterizations of our spaces, starting for the reader's convenience with the (much) simpler situation $\delta =1$. In particular, Theorem~\ref{thm:ThmTwoIntro}. is an effective summary of Theorem~\ref{thm:riesz characterization} and Lemmas~\ref{lem:equivalence Riesz distance}, \ref{lem:nikolskii distance 1} and \ref{lem:nikolskii distance 2}. Section~\ref{sec:ito map} is devoted to establish the local Lipschitz continuity of the It\^{o}-Lyons map in suitable rough path metrics and Theorem~\ref{thm:ThmOneIntro}. can be found in Theorem~\ref{thm:ito map riesz variation} and Corollaries~\ref{cor:ito map riesz type norm} and \ref{cor:ito map nikolskii type norm}. 

\medskip 

\noindent\textbf{Acknowledgment:}
P.K.F. is partially supported by the European Research Council through CoG-683164 and DFG research unit FOR2402. D.J.P. gratefully acknowledges financial support of the Swiss National Foundation under Grant No.~$200021\_163014$. Both authors are grateful for the excellent hospitality of the Hausdorff Research Institute for Mathematics, where the work was initiated.

\section{Riesz type variation}\label{sec:riesz}

In this section we introduce a class of function spaces which unifies the notions of H\"older and $q$-variation regularity. For this purpose we generalize an old version of variation due to F. Riesz and provide two alternative but equivalent characterizations of the so-called Riesz type variation and additionally various embedding results. As explained in the Introduction, the later application in the rough path framework requires us to set up all the function spaces for paths taking values in a metric spaces. \medskip

Let us briefly fix some basic notation: $\mathcal{P}$ is called partition of an interval $[s,t]\subset [0,T]$ if $\mathcal{P}=\{[t_i,t_{i+1}]\,:\, s=t_0 < t_1<\cdots <t_n=t,\, n\in \mathbb{N}\}$. In this case we write $\mathcal{P}\subset [s,t]$ indicating that $\mathcal{P}$ is a partition of the interval $[s,t]$. Furthermore, for such a partition $\mathcal{P}$ and a function $\chi \colon \{ (u,v) \,:\,  s\leq u <v\leq t\}\to \R$ we use the abbreviation
\begin{equation*}
  \sum_{[u,v]\in \mathcal{P}} \chi(u,v):= \sum_{i=0}^{n-1} \chi(t_i,t_{i+1}).
\end{equation*}
If not otherwise specified, $(E,d)$ denote a metric space, $T\in (0,\infty)$ is finite real number and $C([0,T];E)$ stands for the set of all continuous functions $f\colon [0,T] \to E$.\medskip

Two frequently used topologies to measure the regularity of functions are the H\"older continuity and the $q$-variation: 

The \textit{H\"older continuity} of a function $f\in C([0,T];E)$ is measured by 
\begin{equation*}
  \|f\|_{\delta\hol;[s,t]}:=\sup_{u,v\in[s,t],\,u<v}\frac{d(f_u,f_v)}{|v-u|^\delta},\quad  \delta\in (0,1],
\end{equation*}
and $C^{\delta\hol}([0,T];E)$ stands for the set of all functions $f\in C([0,T];E)$ such that $\|f\|_{\delta\hol}:= \|f\|_{\delta\hol;[0,T]} <\infty$. The case $\delta=1$, that is the H\"older continuity of order $1$, is usually refer to as Lipschitz continuity.

The \textit{$q$-variation} of a function $f\in C([0,T];E)$ is defined by
\begin{equation}\label{eq:p-varition}
  \|f\|_{q\var;[s,t]}:=\bigg(\sup_{\mathcal{P}\subset [s,t]} \sum_{[u,v]\in \mathcal{P}} d(f_u,f_v)^q \bigg)^\frac{1}{q},\quad q\in[1,\infty),
\end{equation}
where the supremum is taken over all partitions $\mathcal{P}$ of the interval $[s,t]$. The set of all functions $f\in C([0,T];E)$ with $\|f\|_{q\var}:= \|f\|_{q\var;[0,T]}<\infty$ is denoted by $C^{q\var}([0,T];E)$. The notion of $q$-variation can be traced back to N. Wiener~\cite{Wiener1924}. The special case of $1$-variation is also called bounded variation. A comprehensive list of generalizations of $q$-variation and further references can be found in~\cite{Dudley1999}.

\begin{remark}\label{rmk:besovvar}
  Classical function spaces as fractional Sobolev or more general Besov spaces do not provide a unifying framework simultaneously covering the space of H\"older continuous functions and the space of continuous functions with finite $q$-variation. For example, let us replace for a moment $(E,d)$ by the Euclidean space $(\mathbb{R},|\cdot|)$ and denote the homogeneous Besov spaces by $B^{\delta,p}_{r}([0,T];\mathbb{R})$. While the H\"older space $C^{\delta\hol}([0,T];\R)$ is a special case of Besov spaces, namely the homogeneous Besov space $B^{\delta,\infty}_{\infty}([0,T];\R)$, for $\delta \in (0,1)$, the $q$-variation space $C^{q\var}([0,T];\R)$ is not covered by the wide class of Besov spaces. Indeed, classical embedding theorems, \cite{Young1936} and \cite{Love1938}, yield the following continuous embeddings:
  \begin{equation*}
    B^{\alpha,p}_{\infty}([0,T];\mathbb{R}) \subset C^{p\var} ([0,T];\mathbb{R})\subset B^{1/p,p}_{\infty}([0,T];\mathbb{R}),
  \end{equation*}
  for $p\in (1,\infty)$ and $\alpha \in (1/p,1)$, see also~\cite{Simon1990} and \cite{Friz2006}. In particular, it is known that the second embedding is not an equality. An example can be found in~\cite{Terehin1967}. The relation between the space of functions with finite $q$-variation and Besov spaces was investigated in the literature for a long time, see for example \cite{Musielak1961}, \cite{Peetre1976}, \cite{Bourdaud2006} and \cite{Rosenbaum2009}. For a comprehensive introduction to function spaces we refer to \cite{Triebel2010}.
\end{remark}

To set up a class of function spaces covering precisely and simultaneously the H\"older spaces and the $q$-variation spaces, we introduce a generalized version of a variation due to F. Riesz~\cite{Riesz1910}. For $\delta \in (0,1]$ and $p\in [1/\delta,\infty)$ the \textit{Riesz type variation} of a function $f\in C([0,T];E)$ is given by
\begin{equation}\label{eq:Riesz variation}
  \|f\|_{V^{\delta,p};[s,t]}:= \bigg( \sup_{\mathcal{P}\subset [s,t]} \sum_{[u,v]\in \mathcal{P}} \frac{d(f_u,f_v)^p}{|v-u|^{\delta p-1}}\bigg)^\frac{1}{p},   
\end{equation}
for a subinterval $[s,t]\subset [0,T]$ and for $p=\infty$ we set 
\begin{equation}\label{eq:Riesz variation infty}
  \|f\|_{V^{\delta,\infty};[s,t]} :=\sup_{u,v\in[s,t],\,u<v}\frac{d(f_u,f_v)}{|v-u|^\delta}.
\end{equation}
The set $V^{\delta,p}([0,T];E)$ denotes all continuous functions $f\in C([0,T];E)$ such that $\|f\|_{V^{\delta,p}}:=\|f\|_{V^{\delta,p};[0,T]}<\infty$. The case of $\delta=1$ was originally defined by F. Riesz~\cite{Riesz1910} and a similar generalization as given in \eqref{eq:Riesz variation} was already mentioned in \cite[p.~114,~(14')]{Peetre1976}.

\begin{proposition}\label{prop:riesz interpolation}
  Let $(E,d)$ be a metric space and $T\in (0,\infty)$. For $\delta \in (0,1]$ and $p\in [1/\delta, \infty]$ one has the following relations
  \begin{equation*}
    C^{\delta\hol}([0,T];E) = V^{\delta,\infty}([0,T];E) \subset V^{\delta,p}([0,T];E)\subset V^{\delta,1/\delta}([0,T];E)= C^{1/\delta\var}([0,T];E).
  \end{equation*}
  More precisely, the $1/\delta$-variation of a function $f\in V^{\delta,p}([0,T];E)$ satisfies the bound
  \begin{equation*}
    \|f\|_{1/\delta \var;[s,t]}\leq  \|f\|_{V^{\delta,p} ;[s,t]} |t-s|^{\delta -\frac{1}{p}} 
  \end{equation*}
  for every subinterval $[s,t]\subset [0,T]$. 
\end{proposition}

Before we come to the proof, we need the following remark about super-additive functions.

\begin{remark}\label{rmk:supper additive}
 Setting $\Delta_T:=\{(s,t)\,:\,0\leq s\leq t\leq T\}$ a function $\omega\colon \Delta_T \to [0,\infty)$ is called super-additive if
 \begin{equation*} 
   \omega (s,t) +\omega (t,u) \leq \omega (s,u)\quad \text{for}\quad 0\leq s \leq t \leq u \leq T.   
 \end{equation*}

 Furthermore, if $\omega$ and $\tilde \omega$ are super-additive and $\alpha, \beta >0$ with $\alpha +\beta \geq 1$, then $\omega^\alpha \tilde \omega^\beta$ is super-additive. The proof works as \cite[Exercise~1.8 and 1.9]{Friz2010}.
\end{remark}

\begin{proof}[Proof of Proposition~\ref{prop:riesz interpolation}]
  The identifies 
  \begin{equation*}
    C^{\delta\hol}([0,T];E) = V^{\delta,\infty}([0,T];E)\quad \text{and}\quad V^{\delta,1/\delta}([0,T];E)= C^{1/\delta\var}([0,T];E)
  \end{equation*}
  are ensured by the definitions of the involved function spaces.
  
  The first embedding can be seen by
  \begin{equation*}
    \|f\|_{V^{\delta,p}}^p=  \sup_{\mathcal{P}\subset [0,T]} \sum_{[u,v]\in \mathcal{P}}\bigg( \frac{d(f_u,f_v)}{|v-u|^{\delta}}\bigg)^{p}|v-u|
    \leq T  \|f\|_{C^\delta ;[0,T]}^p , \quad f \in C^\delta ([0,T];E).
  \end{equation*}

  The second embedding is trivial for $\delta=1/p$. For $\delta >1/p$ we first observe that 
  \begin{equation}\label{eq:holder estimate riesz}
    d(f_s,f_t)\leq\bigg( \frac{d(f_s,f_t)^p}{|t-s|^{\delta p -1}} \bigg)^{\frac{1}{p}}|t-s|^{\delta -\frac{1}{p}}\leq \|f\|_{V^{\delta,p};[s,t]}|t-s|^{\delta -\frac{1}{p}},\quad [s,t]\subset [0,T],
  \end{equation}
  for $f\in V^{\delta,p}([0,T];E)$ and thus 
  \begin{equation*}
    d(f_s,f_t)^{\frac{1}{\delta}}\leq \|f\|_{V^{\delta,p};[s,t]}^{\frac{1}{\delta }}|t-s|^{1 -\frac{1}{\delta p}}=:\omega (s,t).
  \end{equation*}
  Since $\|f\|_{V^{\delta,p};[s,t]}^p$ and $|t-s|$ are super-additive as functions in $(s,t)\in \Delta_T$ and $(\delta p)^{-1}+ 1- (\delta p)^{-1}\geq 1$, $\omega$ is a super-additive by Remark~\ref{rmk:supper additive}. Hence, using the super-additivity of $\omega$, we arrive at the claimed estimate 
  \begin{equation*}
    \|f\|_{1/\delta \var;[s,t]}\leq \|f\|_{V^{\delta,p} ;[s,t]} |t-s|^{\delta -\frac{1}{p}}.
  \end{equation*}
\end{proof}

The next lemma justifies the definition of the Riesz type variation in the case of $p=\infty$, cf.~\eqref{eq:Riesz variation infty}, and collects some embedding results of these sets of functions.

\begin{lemma}\label{lem:properties riesz}
  Let $(E,d)$ be a metric space, $T\in (0,\infty)$ and $[s,t]\subset [0,T]$. Suppose $\delta \in (0,1)$ and $p\in [1/\delta,\infty]$. 
  \begin{enumerate}
   \item If $\delta >1/p$, then $V^{\delta,p}([0,T];E)\subset C^{(\delta-1/p)\hol}([0,T];E)$ with the estimate 
         \begin{equation*}
           d(f_s,f_t)\leq \|f\|_{V^{\delta,p};[s,t]}|t-s|^{\delta -\frac{1}{p}},\quad f\in V^{\delta,p}([0,T];E).
         \end{equation*}
   \item If $\delta,\delta^\prime \in (0,1)$ and $p,p^\prime \in [1/\delta,\infty]$ with $\delta^{\prime} < \delta $ and $p^{\prime}< p$, then one has 
         \begin{equation*}
           V^{\delta,p}([0,T];E) \subset V^{\delta^{\prime},p}([0,T];E) \quad \text{and}\quad V^{\delta,p}([0,T];E) \subset V^{\delta,p^{\prime}}([0,T];E)
         \end{equation*}
         with the estimates for $f\in V^{\delta,p}([0,T];E)$
         \begin{equation*}
           \|f\|_{V^{\delta^\prime,p};[s,t]}\leq (t-s)^{\delta-\delta^{\prime}} \|f\|_{V^{\delta,p};[s,t]}
           \quad \text{and}\quad 
           \|f\|_{V^{\delta,p^\prime};[s,t]}\leq (t-s)^{\frac{1}{p^{\prime}}-\frac{1}{p}} \|f\|_{V^{\delta,p};[s,t]}.
         \end{equation*}
   \item For every $f\in V^{\delta,\infty}([0,T];E)$ one has 
         \begin{equation*}
           \lim_{p\to \infty} \|f\|_{V^{\delta,p};[s,t]}=\|f\|_{V^{\delta,\infty};[s,t]}.
         \end{equation*}
  \end{enumerate}
\end{lemma}

\begin{proof}
  (1) The first assertion follows directly by the estimate~\eqref{eq:holder estimate riesz}.
  
  (2) Let $\mathcal{P}$ be a partition of the interval $[s,t]\subset [0,T]$. For $f\in V^{\delta,p}([0,T];E)$ the estimates 
  \begin{equation*}
    \sum_{[u,v]\in \mathcal{P}} \frac{d(f_u,f_v)^p}{|v-u|^{\delta^\prime p -1}} \leq |t-s|^{(\delta -\delta^\prime) p} \sum_{[u,v]\in \mathcal{P}} \frac{d(f_u,f_v)^p}{|v-u|^{\delta p -1}}
  \end{equation*}
  and (using H\"older's inequality)
  \begin{equation*}
    \sum_{[u,v]\in \mathcal{P}} \frac{d(f_u,f_v)^{p^\prime}}{|v-u|^{\delta p^\prime -1}} 
    = \sum_{[u,v]\in \mathcal{P}} \bigg( \frac{d(f_u,f_v)}{|v-u|^{\delta - \frac{1}{p}}}\bigg)^{p^\prime} |v-u|^{1-\frac{p^\prime}{p}}
    \leq |t-s|^{1-\frac{p^\prime}{p}} \bigg(\sum_{[u,v]\in \mathcal{P}} \frac{d(f_u,f_v)^{p}}{|v-u|^{\delta p -1}}\bigg)^{\frac{p^\prime}{p}} 
  \end{equation*}
  lead to (2) by taking the supremum over all partition of $[s,t]$.
  
  (3) Due to Lemma~\ref{lem:properties riesz}~(1), we have 
  \begin{equation*}
    \|f\|_{V^{\delta,\infty};[s,t]}\leq \liminf_{p\to \infty}  \|f\|_{V^{\delta,p};[s,t]},  \quad f\in V^{\delta,\infty}([0,T];E).
  \end{equation*} 
  Furthermore, for $p>q\geq 1$ we get 
  \begin{equation*}
    \|f\|_{V^{\delta,p};[s,t]}
    = \bigg( \sup_{\mathcal{P}\subset [s,t]} \sum_{[u,v]\in \mathcal{P}} \frac{d(f_u,f_v)^q}{|v-u|^{\delta q-1}} \frac{d(f_u,f_v)^{p-q}}{|v-u|^{\delta (p-q)}}\bigg)^\frac{1}{p}
    \leq \|f\|_{V^{\delta,q};[s,t]}^{\frac{q}{p}} \|f\|_{V^{\delta,\infty};[s,t]}^{1-\frac{q}{p}}
  \end{equation*}
  and thus
  \begin{equation*}
    \limsup_{p\to \infty}  \|f\|_{V^{\delta,p};[s,t]}\leq \|f\|_{V^{\delta,\infty};[s,t]}.
  \end{equation*} 
\end{proof}

In the following we introduce two different but equivalent characterizations of the Riesz type variation~\eqref{eq:Riesz variation}. The first one is based on the classical notion of $q$-variation due to Wiener and thus is particularly convenient for applications in rough path theory. The second one relies on certain Besov spaces, namely Nikolskii spaces, which allows to related the Riesz type variation spaces to classical function spaces as fractional Sobolev spaces. See Lemma~\ref{lem:embedding bounded variation} and Theorem~\ref{thm:riesz characterization} for the equivalence. \medskip

In order to give a characterization of Riesz type variation of a function $f\in C([0,T];E)$ in terms of $q$-variation, we introduce a \textit{mixed H\"older-variation regularity} by
\begin{equation}\label{eq:holder-variation mixed norm}
  \|f\|_{\tilde{V}^{\delta,p};[s,t]}:=\bigg( \sup_{\mathcal{P}\subset [s,t]} \sum_{[u,v]\in \mathcal{P}} \frac{\|f\|_{\frac{1}{\delta}\text{-}\mathrm{var};[u,v]}^p}{|v-u|^{\delta p-1}}\bigg)^\frac{1}{p}, \quad \delta \in (0,1], \, p\in [1/\delta,\infty),
\end{equation}
for a subinterval $[s,t]\subset [0,T]$ and in the case of $p=\infty$ we define 
\begin{equation*}
  \|f\|_{\tilde{V}^{\delta,\infty};[s,t]}:= \sup_{\mathcal{P}\subset [s,t]} \sup_{[u,v]\in\mathcal{P}}\frac{\|f\|_{\frac{1}{\delta}\text{-}\mathrm{var};[u,v]}}{|v-u|^{\delta }}.
\end{equation*}
Moreover, we denote by $\tilde V^{\delta, p}([0,T];E)$ the set of all functions $f\in C([0,T];E)$ such that $\|f\|_{\tilde V^{\delta,q}}:=\|f\|_{\tilde V^{\delta,q};[0,T]}< \infty$. \medskip 

An alternative way to measure Riesz type variation of a function $f\in C([0,T];E)$ is related to homogeneous Nikolskii spaces. Hence, we briefly recall the notation of homogeneous \textit{Nikolskii spaces}, which correspond to the homogeneous Besov spaces $B^{\delta,p}_{\infty}([0,T];E)$. For $\delta \in (0,1]$ and $p\in [1,\infty)$ we define
\begin{equation*}
  \|f\|_{N^{\delta,p};[s,t]}:=\sup_{|t-s|\geq h >0} h^{-\delta} \bigg( \int_{s}^{t-h} d (f_u,f_{u+h})^{p}\dd u \bigg)^{\frac{1}{p}}
\end{equation*}
for a subinterval $[s,t]\subset [0,T]$ and for $p=\infty$ we further set 
\begin{equation*}
  \|f\|_{N^{\delta,\infty};[s,t]}:=\sup_{|t-s|\geq h >0} h^{-\delta} \sup_{u\in [s, t-h]}  d(f_u,f_{u+h}).
\end{equation*}
The set of all functions $f\in C([0,T];E)$ such that $\|f\|_{N^{\delta,p}}:=\|f\|_{N^{\delta,p};[0,T]}<\infty$ is denoted by $N^{\delta, p}([0,T];E)$.

Using the definition of Nikolskii regularity, we introduce a \textit{refined Nikolskii type regularity} by
\begin{equation}\label{eq:nikolskii norm integral}
  \|f\|_{\hat N^{\delta,p};[s,t]}:=\bigg( \sup_{\mathcal{P}\subset [s,t]} \sum_{[u,v]\in \mathcal{P}} \|f\|_{N^{\delta,p};[u,v]}^p\bigg)^{\frac{1}{p}}, \quad \delta \in (0,1], \, p\in [1,\infty),
\end{equation}
for $f\in C([0,T];E)$ and a subinterval $[s,t]\subset [0,T]$. For $p=\infty$ we set 
\begin{equation*}
  \|f\|_{\hat N^{\delta,\infty};[s,t]}:= \sup_{\mathcal{P}\subset [s,t]}\sup_{[u,v]\in\mathcal{P}} \|f\|_{N^{\delta,\infty};[u,v]}.
\end{equation*}
Furthermore, $\hat N^{\delta, p}([0,T];E)$ stands for the set of all functions $f \in C([0,T];E)$ such that $\|f\|_{\hat N^{\delta,q}}:=\|f\|_{\hat N^{\delta,q};[0,T]}< \infty$. 

\begin{remark}
  While $\|\cdot\|_{\hat N^{\delta,p};[s,t]}^p$ is a super-additive function in $(s,t)\in\Delta_T$ by its definition, this is not true for the Nikolskii regularity $\|\cdot\|_{N^{\delta,p};[s,t]}^p$ itself if $\delta\in (0,1)$. The later can be seen particularly by Remark~\ref{rmk:counterexample}. 
\end{remark}

In the next two subsections we show that the just introduced two ways of measuring path regularity are indeed equivalent to the Riesz type variation. We start by considering the special case of regularity $\delta=1$, that is the space $V^{1,p}$, in Subsection~\ref{subsec:characterization bounded variation}. The equivalence for general Riesz type variation spaces is content of Subsection~\ref{subsec:characterization riesz variation}

\subsection{Characterization of the space $\bf V^{1,p}$}\label{subsec:characterization bounded variation}

The special case $\delta=1$ or in other words the set $V^{1,p}([0,T];\R^n)$ coincides with the original definition due to F. Riesz~\cite{Riesz1910} and is already fairly well-understood. For the sake of completeness we present here the full picture assuming $E=\R^n$ since it is will be general enough for the later applications concerning the solution map associated to a controlled differential equation, see Subsection~\ref{subsec:Ito map bounded variation}.\medskip 

It is well-known that the Riesz type variation space $V^{1,p}([0,T],\R^n)$ corresponds to the classical Sobolev space $W^{1,p}([0,T];\mathbb{R}^n)$, see e.g. \cite[Proposition~1.45]{Friz2010}. Let us recall the definition of the Sobolev space $W^{1,p}([0,T];\mathbb{R}^n)$ (cf. \cite[Definition~1.41]{Friz2010}). For $p\in [1,\infty]$ and $T\in (0,\infty)$ a function $f\in C([0,T];\mathbb{R}^n)$ is in $ W^{1,p}([0,T];\mathbb{R}^n)$ if and only if $f$ is of the form 
\begin{equation*}
  f_t=f_0+\int_0^t f^\pr_s \dd s,\quad t\in [0,T],
\end{equation*}
for some $f^\pr\in L^p([0,T];\mathbb{R}^n)$. Moreover, we define $\|f\|_{W^{1,p}}:=\|f^\pr\|_{L^p}$ for $f\in W^{1,p}([0,T];\mathbb{R}^n)$. \medskip 

Including the three known characterizations of $V^{1,p}([0,T],\R^n)$, we end up with the following five different ways to measure the Riesz type variation. 

\begin{lemma}\label{lem:embedding bounded variation}
  Let $T\in (0,\infty)$, $p \in (1,\infty)$ and $\mathbb{R}^n$ be equipped with the Euclidean norm $|\cdot|$. The space $ V^{1,p}([0,T];\mathbb{R}^n)$ has the following different characterizations
  \begin{equation*}
    V^{1,p}([0,T];\mathbb{R}^n)=\tilde V^{1,p}([0,T];\mathbb{R}^n)=\hat N^{1,p}([0,T];\mathbb{R}^n) = N^{1,p}([0,T];\mathbb{R}^n)= W^{1,p}([0,T];\mathbb{R}^n) 
  \end{equation*}
  with 
  \begin{equation*}
    \|f\|_{V^{1,p}}=\|f\|_{\tilde V^{1,p}}=\|f\|_{\hat N^{1,p}}=\|f\|_{W^{1,p}} = \|f\|_{N^{1,p}}\quad \text{for} \quad f\in C([0,T];\mathbb{R}^n).
  \end{equation*}
\end{lemma}

\begin{proof}
  For $f\in C([0,T];\mathbb{R}^n)$ and $p\in (1,\infty)$ the identifies
  \begin{equation*}
    \|f\|_{V^{1,p}}=\|f\|_{W^{1,p}} = \|f\|_{N^{1,p}}
  \end{equation*}
  can be found in \cite[Proposition~1.45]{Friz2010} and \cite[Theorem~10.55]{Leoni2009}.
  
  Next we observe that 
  \begin{align*}
    \|f\|_{W^{1,p}}=\|f\|_{V^{1,p}}^p
    \leq  \|f\|_{\tilde V^{1,p}}^p 
    &\leq\sup_{\mathcal{P}\subset [0,T]} \sum_{[u,v]\in \mathcal{P}} \frac{\|f\|_{W^{1,p};[u,v]}^p|v-u|^{p-1}}{|v-u|^{p-1}}
    \leq \|f\|_{W^{1,p}}^p,
  \end{align*}
  where we used \cite[Theorem~1.44]{Friz2010} (see also \cite[Theorem~1]{Friz2006}) for the second estimate and the super-additivity of $\|f\|_{W^{1,p};[u,v]}^p$ as a function in $(u,v)\in \Delta_T$ in the last one.
  
  As last step note that 
  \begin{equation*}
    \|f\|_{\hat N^{1,p}}^p=\sup_{\mathcal{P}\subset [0,T]} \sum_{[u,v]\in \mathcal{P}} \|f\|_{W^{1,p};[u,v]}^p
  \end{equation*}
  due to \cite[Theorem~10.55]{Leoni2009}, which implies
  \begin{equation*}
    \|f\|_{W^{1,p}}^p  = \|f\|_{\hat N^{1,p}}^p \leq \|f\|_{W^{1,p}}^p
  \end{equation*}
  using once more the super-additivity of $\|f\|_{W^{1,p};[u,v]}^p$ as a function in $(u,v)\in \Delta_T$. 
\end{proof}

\subsection{Characterizations of Riesz type variation}\label{subsec:characterization riesz variation}

While Sobolev spaces and Nikolskii spaces coincide with the Riesz type variation spaces for regularity $\delta=1$, this is not true anymore for the fractional regularity $\delta \in (0,1)$. However, the characterizations of Riesz type variation via $q$-variation due to Wiener~\eqref{eq:holder-variation mixed norm} and via classical Nikolskii spaces~\eqref{eq:nikolskii norm integral} still work as we will see in this subsection. \medskip 

We start by recalling the definition of fractional Sobolev spaces. For $\delta \in (0,1)$ and $p\in [1,\infty)$ the \textit{fractional Sobolev} (also called \textit{Sobolev-Slobodeckij}) regularity of a function $f\in C([0,T];E)$ is given by
\begin{equation*}
  \|f\|_{W^{\delta,p};[s,t]}:= \bigg( \iint_{[s,t]^2} \frac{ d (f_u,f_v)^{p}}{|v-u|^{1+\delta p}}\dd u \dd v\bigg)^{\frac{1}{p}}
\end{equation*}
for a subinterval $[s,t]\subset [0,T]$ and we abbreviate $\|\cdot\|_{W^{\delta,p}}:=\|\cdot\|_{W^{\delta,p};[0,T]}$. The set of all functions $f\in C([0,T];E)$ such that $\|f\|_{W^{\delta,p}}<\infty$ is denoted by $W^{\delta,p}([0,T];E)$.\medskip

As an auxiliary result we first need an explicit embedding of Nikolskii regular functions $N^{\delta^{\prime},p}([0,T];E)$ into the set of functions with fractional Sobolev regularity $W^{\delta,p}([0,T];E$).

\begin{lemma}\label{lem:embedding sobolev nikolskii}
  Suppose that $(E,d)$ is a metric space and $ T\in (0,\infty)$. Let $p\in [1,\infty)$ and $\delta, \delta^\pr \in (0,1)$ be such that $\delta^\pr >\delta$. For $f \in N^{\delta^\prime,p}([0,T];E)$ it holds 
  \begin{equation*}
    \|f\|_{W^{\delta,p};[s,t]} \leq \bigg (\frac{2}{(\delta^\pr - \delta )p}\bigg)^\frac{1}{p} \|f\|_{N^{\delta^\pr ,p};[s,t]} (t-s)^{\delta^\pr - \delta}
  \end{equation*}
  for any $s,t\in [0,T]$ with $s<t$. In particular, $N^{\delta^\prime,p}([0,T];E) \subset W^{\delta,p}([0,T];E)$.
\end{lemma}

\begin{proof}
  The fractional Sobolev regularity can be rewritten as
  \begin{equation*}
    \|f\|_{W^{\delta,p};[s,t]}^p = \iint_{[s,t]^2} \frac{ d (f_u,f_v)^{p}}{|v-u|^{1+\delta p}}\dd u \dd v
    = 2 \int_0^{t-s} \int_s^{t-h} \frac{ d (f_u,f_{u+h})^{p}}{|h|^{1+\delta p}}\dd u \dd h
  \end{equation*}
  for $s,t\in [0,T]$ with $s<t$ and for every $f\in W^{\delta,p}([0,T];E)$. Since $f\in N^{\delta^\prime,p}([0,T];E)$, one has
  \begin{equation*}
    \int_s^{t-h}  d (f_u,f_{u+h})^{p} \dd u \leq \|f\|_{N^{\delta^\pr,p};[s,t]}^p h^{\delta^\pr p}.
  \end{equation*}
  Therefore, we conclude for $\delta^\pr > \delta >0$ that 
  \begin{equation*}
    \|f\|_{W^{\delta,p};[s,t]}^p 
    \leq 2 \int_0^{t-s} \frac{\|f\|_{N^{\delta^\pr,p};[s,t]}^p h^{\delta^\pr p}}{|h|^{1+\delta p}} \dd h
    \leq \frac{2}{(\delta^\pr-\delta )p} \|f\|_{N^{\delta^\pr,p};[s,t]}^p (t-s)^{(\delta^\pr -\delta) p },
  \end{equation*}
  for every interval $[s,t]\subset [0,T]$, and thus $N^{\delta^\prime,p}([0,T];E) \subset W^{\delta,p}([0,T];E)$.
\end{proof}

The next proposition presents that functions of refined Nikolskii type regularity are also of finite $q$-variation and H\"older continuous. It can be seen as a refinement of \cite[Theorem~2]{Friz2006}. 

For the sake of notational brevity, we use in the following $A_{\vartheta}\lesssim B_{\vartheta}$, for a generic parameter $\vartheta$, meaning that $A_{\vartheta}\le CB_{\vartheta}$ for some constant $C>0$ independent of $\vartheta$.

\begin{proposition}\label{prop:variation embeddings}
  Suppose that $(E,d)$ is a metric space and $T\in (0,\infty)$. Let $\delta \in (0,1)$ and $p\in (1,\infty)$ be such that $\alpha := \delta - 1/p>0$, and set $q:=\frac{1}{\delta}$.  
  \begin{enumerate}
    \item If $f\in N^{\delta,p}([0,T];E)$, then $f\in C^{\alpha\hol}([0,T];E)$ and 
          \begin{equation*}
            d(f_s,f_t) \lesssim  \|f\|_{N^{\delta ,p};[s,t]}  (t-s)^{\delta -\frac{1}{p}},\quad [s,t]\subset [0,T].
          \end{equation*}
    \item The $q$-variation of any $f\in \hat N^{\delta, p}([0,T];E)$ can be estimated by 
          \begin{equation*}
            \|f\|_{q\var; [s,t]} \lesssim \|f\|_{\hat N^{\delta,p};[s,t]} (t-s)^{\alpha},\quad [s,t]\subset [0,T],
          \end{equation*}
          and one has $\hat N^{\delta,p}([0,T];E)\subset C^{\alpha\hol}([0,T];E)$ and $\hat N^{\delta,p}([0,T];E) \subset C^{q\var}([0,T];E)$.
  \end{enumerate}
\end{proposition}

\begin{proof}
  (1) Choose $\gamma < \delta$ such that $\gamma - 1/p>0$. Because $f\in N^{\delta,p}([0,T];E)$, Lemma~\ref{lem:embedding sobolev nikolskii} yields $f\in W^{\gamma,p}([0,T];E)$ and we have 
  \begin{equation*}
    \|f\|_{W^{\gamma,p};[s,t]}^p = F_{s,t} :=\iint_{[s,t]^2} \bigg(\frac{d(f_u,f_v)}{|v-u|^{1/p+\gamma}} \bigg)^p\dd u \dd v, \quad [s,t]\subset [0,T].
  \end{equation*}
  Applying the Garsia-Rodemich-Rumsey inequality with $\Psi(\cdot) =(\cdot)^p$ and $p(\cdot)=(\cdot)^{1/p+\gamma}$ gives
  \begin{align*}
    d(f_s,f_t)\leq 8 \int_0^{t-s}\bigg( \frac{F_{s,t}}{u^2} \bigg)^{\frac{1}{p}}\dd p(u) = \frac{8}{(\gamma-1/p)} \|f\|_{W^{\gamma,p};[s,t]} (t-s)^{\gamma-\frac{1}{p}},
  \end{align*}
  using $\gamma-\frac{1}{p}>0$, see for instance \cite[Theorem~A.1]{Friz2010} for a version of the Garsia-Rodemich-Rumsey lemma suitable for functions with values in a metric space. Furthermore, Lemma~\ref{lem:embedding sobolev nikolskii} yields
  \begin{align}\label{eq:holder estimate proof}
  \begin{split}
    d(f_s,f_t)
    &\leq \frac{8}{(\gamma-1/p)} \bigg (\frac{2}{(\delta - \gamma )p}\bigg)^\frac{1}{p} \|f\|_{N^{\delta ,p};[s,t]} (t-s)^{\delta - \gamma} (t-s)^{\gamma-\frac{1}{p}}\\
    &\lesssim  \|f\|_{N^{\delta ,p};[s,t]}  (t-s)^{\delta -\frac{1}{p}},
  \end{split}
  \end{align}
  which gives $f\in C^{(\delta-1/p)\hol}([0,T];E)$. 
  
  (2) Assuming $f\in \hat N^{\delta,p}([0,T];E)$ the estimate~\eqref{eq:holder estimate proof} leads to 
  \begin{equation*}
    d(f_s,f_t)\lesssim  \|f\|_{\hat N^{\delta ,p};[s,t]}  (t-s)^{\delta -\frac{1}{p}},  \quad [s,t]\subset [0,T].
  \end{equation*}
  Recalling $\alpha = \delta - 1/p>0$ and $q=\frac{1}{\delta}$, we note that 
  \begin{equation*}
    \omega (s,t):= \|f\|^{q}_{\hat N^{\delta,p};[s,t]} (t-s)^{\alpha q},\quad 0\leq s\leq t\leq T,
  \end{equation*}
  is super-additive. Indeed, since $\|f\|^{p}_{\hat N^{\delta,p};[s,t]}$ and $|t-s|$ are super-additive as functions in $(s,t)\in \Delta_T$ and $q/p+1-1/\delta p \geq  1$, Remark~\ref{rmk:supper additive} ensures the super-additivity of $\omega$.
   
  Hence, by the super-additivity of $\omega$ we deduce that 
  \begin{equation*}
    \|f\|_{q\text{-var};[s,t]}^q \leq C \omega (s,t)= C \|f\|^{q}_{\hat N^{\delta,p};[s,t]}|t-s|^{\alpha q},
  \end{equation*}
  for some constant $C>0$ depending only on $\delta$ and $p$.
  
  In particular, we have proven that $\hat N^{\delta,p}([0,T];E)\subset C^{\alpha\hol}([0,T];E)$ and $\hat N^{\delta,p} ([0,T];E)\subset C^{q\var}([0,T];E)$.
\end{proof}

\begin{remark}
  Proposition~\ref{prop:variation embeddings}~(2) does not hold for $\|\cdot\|_{\hat N^{\delta,p};[s,t]}$ replaced by $\|\cdot\|_{N^{\delta,p};[s,t]}$, see Remark~\ref{rmk:counterexample} below.
\end{remark}

\begin{remark} 
  Alternatively to the given proofs of Lemma~\ref{lem:embedding sobolev nikolskii} and Proposition~\ref{prop:variation embeddings}, one could use the abstract Kuratowski embedding to extend the known Besov embeddings from Banach spaces to general metric spaces and then proceed further as presented above. For example note, if $(E,\|\cdot \|)$ is a Banach space, then classical Besov embeddings lead to 
  \begin{equation*}
    \|f_t-f_s\| \leq \sup_{|t-s|\geq h >0 } \bigg( \frac{\|f_{s+h}-f_s\|}{|h|^{\delta -1/p}} \bigg) |t-s|^{\delta -1/p}\leq C \|f\|_{N^{\delta,p};[s,t]} |t-s|^{\delta -1/p},
  \end{equation*}
  for every $f\in N^{\delta,p}([0,T];E)$, $\delta \in (0,1)$, $p\in (1,\infty)$ such that $\delta >1/p$, and some constant $C>0$,  cf. ~\cite[Theorem~10]{Simon1990}.  However, we preferred here to give direct proofs. 
  
  On the other hand, the embedding $N^{\delta, p}([0,T];E)\subset  W^{\delta,p}([0,T];E)$ does not hold true, which prevents to deduce Proposition~\ref{prop:variation embeddings} as a corollary of \cite[Theorem~2]{Friz2006}. Hence, the elaborated embedding of Lemma~\ref{lem:embedding sobolev nikolskii} is essential to obtain Proposition~\ref{prop:variation embeddings}.
\end{remark}

The next theorem is the main result of the first part: the characterization of Riesz type variation via $\|\cdot \|_{\tilde V^{\delta,p}}$ and $\|\cdot \|_{\hat N^{\delta,p}}$. 

\begin{theorem}\label{thm:riesz characterization}
  Let $T\in (0,\infty)$ and $(E,d)$ be a metric space. Suppose that $\delta \in (0,1)$ and $p \in (1,\infty)$ such that $\delta > 1/p$. Then, $\|\cdot\|_{V^{\delta,p}}$, $\|\cdot\|_{\tilde V^{\delta,p}}$ and $\|\cdot\|_{\hat{N}^{\delta,p}}$ are equivalent, that is 
  \begin{equation*}
    \|f\|_{ V^{\delta,p}} \lesssim \|f\|_{\tilde V^{\delta,p}} \lesssim \|f\|_{\hat N^{\delta,p}} \lesssim \|f\|_{ V^{\delta,p}}
  \end{equation*}
  for every function $f\in C([0,T];E)$, and thus 
  \begin{equation*}
    V^{\delta,p}([0,T];E)= \tilde V^{\delta,p}([0,T];E)= \hat N^{\delta,p}([0,T];E).
  \end{equation*}
\end{theorem}

\begin{proof}
  For a function $f\in C([0,T];E)$ and an interval $[s,t]\subset [0,T]$ recall that 
  \begin{equation*}
    \|f\|_{N^{\delta,p};[s,t]}= \bigg( \sup_{|t-s| \geq h >0}h^{-\delta p}\int_s^{t-h} d(f_{u},f_{u+h})^{p}\dd u \bigg)^{\frac{1}{p}}.  
  \end{equation*}
  Let us fix $h\in (0,t-s]$ and take a partition $\mathcal{P}(h):=\{[t_i,t_{i+1}]\,:\, s=t_0<\dots<t_M=t-h\}$ such that
  \begin{equation*}
    |t_{M}-t_{M-1}|\leq h\quad \text{and} \quad |t_{i+1}-t_i|=h \quad \text{for } i=0,\dots, M-2, \quad M\in \mathbb{N}.     
  \end{equation*}
  Since $\sup_{u \in [t_i,t_{i+1}]} d(f_u,f_{u+h})^p\leq \|f\|^p_{1/\delta\var;[t_i,t_{i+2}]}$ for $ i=0,\dots, M-1$ with $t_{M+1}:=t-h$, we observe that 
  \begin{align*}
    \int_s^{t-h} |d(f_{u},f_{u+h})|^{p}\dd u &= \sum_{i=0}^{M-1} \int_{t_i }^{t_{i+1}} d(f_{u},f_{u+h})^p \dd u
    \leq  \sum_{i=0}^{M-1} \sup_{u \in [t_i,t_{i+1}]} d(f_{u},f_{u+h})^p (t_{i+1}-t_i) \\
    &\leq \frac{1}{2}(2h)^{\delta p} \sum_{i=0}^{M-1} \frac{\|f\|^p_{1/\delta\var;[t_i,t_{i+2}]}}{(2h)^{\delta p-1}}\lesssim h^{\delta p} \|f\|_{\tilde V^{\delta,p};[s,t]}^p,
  \end{align*}
  which implies $\|f\|_{N^{\delta,p};[s,t]}^p\lesssim\|f\|_{\tilde V^{\delta,p};[s,t]}^p$. Therefore, the super-additivity of $\|f\|_{\tilde V^{\delta,p};[s,t]}^p$ as function in $(s,t)\in \Delta_T$ reveals 
  \begin{equation*}
    \|f\|_{\hat N^{\delta,p}} \lesssim \|f\|_{\tilde V^{\delta,p}}.
  \end{equation*} 

  For the converse inequality Proposition~\ref{prop:variation embeddings} gives  
  \begin{equation*}
    \|f\|_{\frac{1}{\delta}\text{-}\mathrm{var};[u,v]} \lesssim \|f\|_{\hat N^{\delta,p};[u,v]} |v-u|^{\delta - \frac{1}{p}}, \quad 0\leq u< v\leq T,
  \end{equation*}
  for $\delta \in (0,1)$ and $p \in (1,\infty)$ such that $\delta > 1/p$, which leads to 
  \begin{equation*}
    \|h\|_{\tilde{V}^{\delta,p}}^p = \sup_{\mathcal{P}\subset [0,T]} \sum_{[u,v]\in \mathcal{P}} \frac{\|h\|_{\frac{1}{\delta}\text{-}\mathrm{var};[u,v]}^p}{|v-u|^{\delta p- 1}}
    \leq \sup_{\mathcal{P}\subset [0,T]} \sum_{[u,v]\in \mathcal{P}} \frac{\|h\|_{\hat N^{\delta,p};[u,v]}^p|v-u|^{\delta p-1}}{|v-u|^{\delta p-1}}\leq \|h\|_{\hat N^{\delta,p}}^p,
  \end{equation*}
  where we applied the super-additivity of $\|f\|_{\hat N^{\delta;p};[s,t]}^p$ as a function in $(s,t)\in \Delta_T$.
  
  It remains to show
  \begin{equation*}
    \|f\|_{ V^{\delta,p}} \lesssim \|f\|_{\tilde V^{\delta,p}} \lesssim  \|f\|_{ V^{\delta,p}},\quad f\in C([0,T],E).
  \end{equation*}
  The first inequality follows immediately from the definitions and the observation 
  \begin{equation*}
    d(f_u,f_v)^p\leq \|f\|_{1/\delta\var;[u,v]}^p,\quad [u,v]\subset [0,T].
  \end{equation*}
  The second inequality can be deduced from Proposition~\ref{prop:riesz interpolation}, which gives the estimate 
  \begin{equation*}
    \|f\|_{1/\delta\var;[u,v]}^p \leq  \|f\|^p_{V^{\delta,p};[u,v]} |v-u|^{\delta p-1},
  \end{equation*} 
  and the super-additivity of $\|f\|_{V^{\delta;p};[s,t]}^p$ as a function in $(s,t)\in \Delta_T$.
\end{proof}

As a next step we briefly want to understand how the set $V^{\delta,p}([0,T];E)$ of functions with finite Riesz type variation are related to other types of measuring the regularity of functions. The characterization of Riesz type variation in terms of Nikolskii regularity allows to deduce the following result connecting the set $V^{\delta,p}([0,T];E)$ with the notion of classical fractional Sobolev and Nikolskii regularity.

\begin{corollary}\label{cor:212}
  Let $T\in (0,\infty)$ and $(E,d)$ be metric space. If $\delta \in (0,1)$ and $p \in (1,\infty)$ such that $\delta > 1/p$, then one has the inclusions 
  \begin{equation}\label{eq:embedding riesz}
    W^{\delta,p}([0,T];E)\subset V^{\delta,p}([0,T];E)\subset N^{\delta,p}([0,T];E)
  \end{equation}
  and 
  \begin{equation*}
    N^{\delta+\epsilon,p}([0,T];E)\subset \hat N^{\delta,p}([0,T];E)\subset N^{\delta,p}([0,T];E)
  \end{equation*}
  for $\epsilon\in (0,1-\delta)$.
\end{corollary}

\begin{proof}
  For the first embedding let $f\in W^{\delta,p}([0,T];E)$. Applying Theorem~\ref{thm:riesz characterization} and \cite[Theorem~11]{Simon1990}, which can be extended to general metric spaces by Kuratowski's embedding theorem, we get
  \begin{equation*}
    \|f\|_{V^{\delta,p}}^p\lesssim \|f\|_{\hat N^{\delta,p}}^p
    = \sup_{\mathcal{P}\subset [0,T]} \sum_{[s,t]\in \mathcal{P}}\|f\|_{N^{\delta,p};[s,t]}^p 
    \lesssim \sup_{\mathcal{P}\subset [0,T]} \sum_{[s,t]\in \mathcal{P}}\|f\|_{W^{\delta,p};[s,t]}^p 
    \leq \|f\|_{W^{\delta,p}}^p.
  \end{equation*}
  
  For the second embedding let $f\in N^{\delta,p}([0,T];E)$ and we apply again Theorem~\ref{thm:riesz characterization} to obtain 
  \begin{equation*}
    \|f\|_{N^{\delta,p}}^p\leq \|f\|_{\hat N^{\delta,p}}^p \lesssim \|f\|_{V^{\delta,p}}^p.
  \end{equation*}
  
  The first embedding for the refined Nikolskii type space $\hat N^{\delta+\epsilon,p}$ is a consequence of Theorem~\ref{thm:riesz characterization} and the embedding 
  \begin{equation*}
     N^{\delta+\epsilon,p}([0,T];E) \subset W^{\delta,p}([0,T];E) \subset V^{\delta,p}([0,T];E), 
  \end{equation*}
  where we used Lemma~\ref{lem:embedding sobolev nikolskii} and \eqref{eq:embedding riesz}.
  
  The second embedding for the refined Nikolskii type space $\hat N^{\delta,p}$ follows directly from its definition.
\end{proof}

\begin{remark}\label{rmk:counterexample}
  Both embeddings are proper embeddings, which means in both cases the equality does not hold. 
  
  Indeed, an example of a set of functions which are included in $V^{1/2+H,2}([0,T];\R)$ but not in $W^{1/2+H,2}([0,T];\R)$ consists of the Cameron-Martin space of a fractional Brownian motion with Hurst index $H\in (0,1/2)$, see \cite{Friz2006} and \cite[Section~11]{Friz2014} and the references therein.
  
  To see that the second embedding is not an equality, we recall that the sample paths of a Brownian motion belong the Nikolskii space $N^{1/2,p}([0,T];\R)$ for $p\in (2,\infty)$, which was proven by~\cite{Roynette1993} (cf.~\cite[Proposition~1]{Rosenbaum2009}). However, it is also well-known that sample paths of a Brownian motion are not contained in $C^{2\var}([0,T];\R)$. In other words, they cannot be contained in $V^{1/2,p}([0,T];\R)$ for $p\in (2,\infty)$ since this is a subset of $C^{2\var}([0,T];\R)$ by Proposition~\ref{prop:riesz interpolation}.
\end{remark}

\subsection{Separability considerations}

In order to embed the Riesz type variation spaces into separable Banach spaces, we need to restricted the general metric space $E$ and focus here on the case $E=\R^n$ equipped with the Euclidean norm $|\cdot |$. As usual $|\cdot|$ induces the metric $d(x,y):= |y-x|$ for $x,y\in \R^n$ and thus $\|\cdot\|_{V^{\delta, p}}$, $\|\cdot\|_{\tilde V^{\delta, p}}$ and $\|\cdot\|_{\hat N^{\delta, p}}$ become semi-norms, which can be easily modified to proper norms by adding for instance the Euclidean norm of the functions evaluated at zero, cf. \eqref{eq:norm}. An immediate consequence of Theorem~\ref{thm:riesz characterization} is the following equivalence.

\begin{corollary}
  Let $T\in (0,\infty)$ and $\mathbb{R}^n$ be equipped with the Euclidean norm $|\cdot|$. If $\delta \in (0,1)$ and $p \in (1,\infty)$ are such that $\delta > 1/p$, then the semi-norms $\|\cdot\|_{V^{\delta,p}}$, $\|\cdot\|_{\tilde V^{\delta,p}}$ and $\|\cdot\|_{\hat{N}^{\delta,p}}$ are equivalent. 
\end{corollary}

In order to turn $C^{\delta\hol}([0,T];\R^n)$ and $C^{p\var}([0,T];\R^n)$ into Banach spaces, one usually introduces the norms 
\begin{equation}\label{eq:norm}
  |f(0)|+\|f\|_{\delta\hol}\quad \text{and}\quad |g(0)|+\|g\|_{p\var}
\end{equation}
for $f\in C^{\delta\hol}([0,T];\R^n)$ and $g \in C^{p\var}([0,T];\R^n)$, respectively. These Banach spaces are not separable, see~\cite[Theorem~5.25]{Friz2010}.

To restore the separability, one can consider the closure of smooth paths. Let $C^\infty([0,T];\R^n)$ be the space of smooth functions $f\in C([0,T];\R^n)$. For $\delta \in (0,1)$ and $p\in (1,\infty)$ we define
\begin{align*}
  C^{0, \delta\hol}([0,T];\R^n):=\overline{C^\infty([0,T];\R^n)}^{\|\cdot\|_{\delta\hol}} \text{ and }
  C^{0, p\var}([0,T];\R^n):=\overline{C^\infty([0,T];\R^n)}^{\|\cdot\|_{p\var}}.
\end{align*}
These two Banach spaces are separable and one has the obvious embeddings 
\begin{equation*}
  C^{0, \delta\hol}([0,T];\R^n)\subset C^{\delta\hol}([0,T];\R^n)\quad \text{and}\quad 
  C^{0, p\var}([0,T];\R^n)\subset C^{p\var}([0,T];\R^n).
\end{equation*}

The Riesz type variation space $V^{\delta ,p}([0,T];\R^n)$ can be embedded into $ C^{0, \alpha \hol}([0,T];\R^n)$ and $C^{0, p\var}([0,T];\R^n)$.

\begin{lemma}
  Let $T\in (0,\infty)$ and $\mathbb{R}^n$ be equipped with the Euclidean norm $|\cdot|$. If $\delta \in (0,1)$ and $p \in (1,\infty)$ are such that $\delta > 1/p$, then one has the embeddings 
  \begin{equation*}
    V^{\delta,p}([0,T];\R^n)\subset C^{0, p\var}([0,T];\R^n)\quad \text{and}\quad V^{\delta,p}([0,T];\R^n)\subset C^{0, \alpha \hol}([0,T];\R^n) 
  \end{equation*}
  for $\alpha \in (0,\delta-1/p)$.
\end{lemma}

\begin{proof}
  For $f\in V^{\delta,p}([0,T];\R^n)$ and $\delta > 1/p$ we apply Lemma~\ref{lem:properties riesz} to obtain
  \begin{align*}
    \lim_{\varepsilon \to 0} \sup_{\mathcal{P}\subset [0,T],\,|\mathcal{P}|<\varepsilon} \sum_{[s,t]\in \mathcal{P}} |f_t-f_s|^p
    &\lesssim  \lim_{\varepsilon \to 0} \sup_{\mathcal{P}\subset [0,T],\,|\mathcal{P}|<\varepsilon} \sum_{[s,t]\in \mathcal{P}} \|f\|_{V^{\delta,p};[s,t]}^p|t-s|^{\delta p -1}\\
    &\leq \|f\|_{V^{\delta,p}}^p \lim_{\varepsilon \to 0} \varepsilon^{\delta p -1}=0,
  \end{align*}
  where $|\mathcal{P}|$ denotes the mesh size of the partition $\mathcal{P}$, and thus $f\in C^{0, p\var}([0,T];\R^n)$ due to Wiener's characterization of $C^{0,p\var}([0,T];\R^n)$, see \cite[Theorem~5.31]{Friz2010}.
  
  Using Wiener's characterization of $C^{0,\alpha \hol}([0,T];\R^n)$ for $\alpha \in (0,\delta-1/p)$, we get the second embedding because of
  \begin{equation*}
    \lim_{\varepsilon \to 0} \sup_{[s,t]\subset [0,T],\,|t-s|<\varepsilon} \frac{|f_t-f_s|}{|t-s|^\alpha}
    \leq \lim_{\varepsilon \to 0} \sup_{[s,t]\subset [0,T],\,|t-s|<\varepsilon} \bigg ( \frac{|f_t-f_s|^p}{|t-s|^{\delta p-1}}\bigg)^{1/p} \varepsilon^{\delta-1/p - \alpha}=0
  \end{equation*}
  for $f\in V^{\delta,p}([0,T];\R^n)$.
\end{proof}

\section{Continuity of the It\^o-Lyons map}\label{sec:ito map}

The dynamics of a controlled differential equation driven by a path $X\colon [0,T]\to \mathbb{R}^n$ of finite $q$-variation is formally given by 
\begin{equation}\label{eq:differential equation}
  \d Y_t = V(Y_t)\dd X_t,\quad Y_0=y_0, \quad t\in [0,T],
\end{equation}
where $y_0 \in \mathbb{R}^m$ is the initial condition, $V\colon \mathbb{R}^m \to \mathcal{L}(\mathbb{R}^n,\mathbb{R}^m)$ is a smooth enough vector field and $T\in (0,\infty)$. Here $\mathcal{L}(\mathbb{R}^n,\mathbb{R}^m)$ denotes the space of linear operators from $\mathbb{R}^n$ to $\mathbb{R}^m$. If the driving signal $X\in C^{p\var}([0,T];\mathbb{R}^n)$ for $p\in [1,2)$, Lyons~\cite{Lyons1994} first established the existence and uniqueness of a solution $Y$ to the equation~\eqref{eq:differential equation}. Moreover, he proved that the It\^o-Lyons map is a locally Lipschitz continuous map with respect to the $q$-variation topology. In order to restore the continuity for more irregular paths $X$, say $X\in C^{q\var}([0,T];\mathbb{R}^n)$ for an arbitrary large $q<\infty$, Lyons introduced the notion of rough paths in his seminal paper~\cite{Lyons1998}, see Subsection~\ref{subsection:rough path}. Based on Lyons' estimate, one can deduce the local Lipschitz continuity of the It\^o-Lyons map with respect to a H\"older topology, see for example~\cite{Friz2005a}. \medskip

The aim of this section is to particularly unify these two results by establishing the local Lipschitz continuity of the It\^o-Lyons map on Riesz type variation spaces. For this purpose we combine Lyons' estimates with our characterization of Riesz type variation in terms to $q$-variation to deduce the locally Lipschitz continuity of the It\^o-Lyons map with respect to an inhomogeneous Riesz type distance. See Proposition~\ref{eq:ito map bounded variation} for the continuity result in the regime of bounded variation paths. For the result in the general rough path setting we refer to Theorem~\ref{thm:ito map riesz variation} and the Corollaries~\ref{cor:ito map riesz type norm} and \ref{cor:ito map nikolskii type norm}. \medskip

To quantify the regularity of the vector field $V$ in the controlled differential equation~\eqref{eq:differential equation}, we introduce for $\alpha >0$ the space $\Lip^\alpha:=\Lip^\alpha ( \mathbb{R}^m;\mathcal{L}(\mathbb{R}^n,\mathbb{R}^m))$ in the sense of E. Stein, cf. \cite[Definition~10.2]{Friz2010}. For  $\alpha>0$ and $\lfloor \alpha \rfloor:= \max \{n\in \mathbb{N}\,:\, n \leq \alpha \}$ the space $\Lip^\alpha$ consists of all maps $V\colon \mathbb{R}^m \to \mathcal{L}(\mathbb{R}^n,\mathbb{R}^m)$ such that $V$ is $\lfloor \alpha \rfloor$-times continuously differentiable with $(\alpha - \lfloor \alpha \rfloor)$-H\"older continuous partial derivatives of order $\lfloor \alpha \rfloor$ (or with continuous partial derivatives of order $\alpha$ in the case $\alpha = \lfloor \alpha \rfloor$). On the space $\Lip^\alpha$ we introduce the usual norm $\| \cdot \|_{\Lip^\alpha}$ and further denote the supremum norm by $\|\cdot \|_\infty$. For the supremum norm on $C([0,T];\mathbb{R}^n)$ we write $\|\cdot\|_{\infty;[0,T]}:=\sup_{0\leq t\leq T}|\cdot|$.

\subsection{Continuity w.r.t. $\bf \tilde V^{1,p}$}\label{subsec:Ito map bounded variation}

In this subsection we derive the local Lipschitz continuity of the solution map on the Riesz type variation spaces $V^{1,p}([0,T];\R^n)$. To that end the equivalent characterization of $V^{1,p}([0,T];\R^n)$ given by $\tilde V^{1,p}([0,T];\R^n)$ turns out to be particularly convenient. The solution map~$\Phi$ is defined by
\begin{equation}\label{eq:ito map simple}
  \Phi \colon  \mathbb{R}^m \times \Lip^1 \times \tilde V^{1,p}([0,T];\mathbb{R}^n)  \to \tilde V^{1,p}([0,T];\mathbb{R}^m) \quad \text{via} \quad \Phi(y_0,V,X):=Y,
\end{equation}
where $Y$ denotes the solution to the integral equation
\begin{equation}\label{eq:rde integral}
  Y_t = y_0 +\int_0^t V(Y_s)\dd X_s, \quad t\in [0,T].
\end{equation}
First notice that the integral appearing in equation~\eqref{eq:rde integral} can be defined as a classical Riemann-Stieltjes integral with respect to bounded variation functions because of the embedding $\tilde V^{1,p}([0,T];\mathbb{R}^n)\subset C^{1\var}([0,T];\mathbb{R}^n)$ for all $p\in (1,\infty)$ due to Proposition~\ref{prop:riesz interpolation} and Lemma~\ref{lem:embedding bounded variation}. 

\begin{proposition}\label{eq:ito map bounded variation}
  For $X\in \tilde V^{1,p}([0,T];\mathbb{R}^n)$ with $p \in (1,\infty )$, $V\in \Lip^1$ and every initial condition $y_0 \in \mathbb{R}^m$, the controlled differential equation~\eqref{eq:rde integral} has a unique solution $Y\in \tilde V^{1,p}([0,T];\mathbb{R}^n)$ and the solution map~$\Phi$ as defined in~\eqref{eq:ito map simple} is locally Lipschitz continuous.
  
  More precisely, for $y_0^i \in \mathbb{R}^m$, $X^i\in \tilde V^{1,p}([0,T];\mathbb{R}^n)$, $V^i\in \Lip^1$ such that 
  \begin{equation*}
    \|X^i\|_{\tilde V^{1,p}}\leq b \quad \text{and} \quad \|V^i\|_{\Lip^1}\leq l,\quad i=1,2,
  \end{equation*}
  for some $b,l>0$ and corresponding solution $Y^i$, there exist a constant $C=C(b,l,p)\geq 1$ such that 
  \begin{equation*}
    \|Y^1-Y^2\|_{\tilde V^{1,p}} \leq C \big(\|V^1-V^2\|_{\infty} + |y_0^1-y^2_0|+ \|X^1- X^2\|_{\tilde V^{1,p}}\big).
  \end{equation*}
\end{proposition}

\begin{proof}
  Since $X^i\in \tilde V^{1,p}([0,T];\mathbb{R}^n)$, $X^i$ is in particular of bounded variation and thus the integral equation~\eqref{eq:rde integral} is well-defined and admits a unique solution $Y^i\in C^{1\var}([0,T];\mathbb{R}^n)$ for each $i=1,2$. Moreover, for every subinterval $[s,t]\subset [0,T]$ the local Lipschitz continuity of the solution map~$\Phi$ in $1$-variation, c.f.~\cite[Theorem~3.18]{Friz2010} and \cite[Remark~3.19]{Friz2010}, yields
  \begin{equation*}
    \|Y^1-Y^2\|_{1\var;[s,t]}\leq 2 \exp(3bl)\big(|y^1_s-y^2_s| l c(s,t) + \|V^1-V^2\|_\infty c(s,t)+l \| X^1-X^2\|_{1\var;[s,t]} \big) 
  \end{equation*}
  where $c(s,t)$ can be chosen such that 
  \begin{equation*}
    \|X^1\|_{1\var;[s,t]}+ \|X^2\|_{1\var;[s,t]}\leq c(s,t) \lesssim k(s,t) (t-s)^{1-1/p}
  \end{equation*}
  with $k(s,t):= \|X^1\|_{\tilde V^{1,p};[s,t]}+ \|X^2\|_{\tilde V^{1,p};[s,t]}$.
  Dividing both sides by $|t-s|^{1-1/p}$ and taking them to the power $p$ leads to 
  \begin{align*}
    &\frac{\|Y^1-Y^2\|^p_{1\var;[s,t]}}{|t-s|^{p-1}}\\
    &\qquad \lesssim \exp(3blp)\bigg(|y^1_s-y^2_s|^p l^p k(s,t)^p + \|V^1-V^2\|_\infty^p k(s,t)^p+l^p \frac{\|X^1-X^2\|_{1\var;[s,t]}^p}{|t-s|^{p-1}} \bigg).
  \end{align*}
  From this inequality we deduce, by summing over a partition of $[0,T]$ and taking then the supremum over all partitions, that
  \begin{equation*}
    \|Y^1-Y^2\|_{ \tilde V^{1,p}} \lesssim \exp(3vl) \big (\|V^1-V^2\|_{\infty}b  + \|Y^1-Y^2\|_{\infty;[0,T]}b l + l \|X^1- X^2\|_{\tilde V^{1,p}}\big),
  \end{equation*}
  where we used the super-additivity of $k(s,t)^p$ and $k(0,T)\leq 2 b$. Finally, $\|Y^1-Y^2\|_{\infty;[0,T]}$ can be estimated by \cite[Theorem~3.15]{Friz2010} to complete the proof.
\end{proof}

\begin{remark}
  An immediate consequence of Lemma~\ref{lem:embedding bounded variation} is that the local Lipschitz continuity as stated in Proposition~\ref{eq:ito map bounded variation} of the solution map~$\Phi$ given by~\eqref{eq:ito map simple}  also holds with respect to the (equivalent) Sobolev or Nikolskii metric.
\end{remark}

\subsection{Continuity w.r.t. general Riesz type variation}\label{subsection:rough path}

In order to give a meaning to the controlled differential equation~\eqref{eq:differential equation} for driving signals $X$ which are not of bounded variation, we introduce here the basic framework of rough path theory. For more comprehensive monographs about rough path theory we refer to \cite{Lyons2002,Friz2010,Friz2014}, and for the convenience of the reader the following definitions are mainly borrowed from~\cite{Friz2010}. \medskip

As already explained in the Introduction, a rough path takes values in the metric space $(G^N(\mathbb{R}^n),d_{cc})$ and not ``only'' in the Euclidean space $(\R^n,\|\,\cdot\,\|)$. Let us recall the basic ingredients: 

For $N\in\mathbb{N}$ and a path $Z\in C^{1\var}(\mathbb{R}^n)$ its $N$-step signature is given by 
\begin{align*}
  S_N(Z)_{s,t}:=&\bigg (1, \int_{s<u<t}\dd Z_u, \dots, \int_{s<u_1<\dots <u_N<t}\dd Z_{u_1} \otimes \cdots \otimes \d Z_{u_N} \bigg) \\
  &\in T^N(\mathbb{R}^n):= \bigoplus_{k=0}^N \big(\mathbb{R}^n\big)^{\otimes k},
\end{align*}
where $\big(\mathbb{R}^n\big)^{\otimes k}$ denotes the $k$-tensor space of $\mathbb{R}^n$ and $\mathbb{R}^{\otimes 0}:= \mathbb{R}$. We note that $T^N(\mathbb{R}^n)$ is an algebra (``level-$N$ truncated tensor algebra'') under the tensor product $\otimes$. The corresponding space of all these lifted paths is the step-$N$ free nilpotent group (w.r.t. $\otimes$) 
\begin{equation*}
  G^N(\mathbb{R}^n):= \{S_N(Z)_{0,T} \,:\, Z\in C^{1\var}([0,T];\mathbb{R}^n)\}\subset T^N(\mathbb{R}^n).
\end{equation*}
For every $g\in G^N(\mathbb{R}^n)$ the so-called ``Carnot-Caratheodory norm''
\begin{equation*}
  \|g\|_{cc}:= \inf \bigg\{  \int_0^T\, \|\d \gamma_s \| \,:\, \gamma\in C^{1\var}([0,T];\R^n) \text{ and } S_N(\gamma)_{0,T}=g \bigg \},
\end{equation*}
where $ \int_0^T\, \|\d \gamma_s \|$ is the length of $\gamma$ based on the Euclidean distance, is finite and the infimum is attained, see~\cite[Theorem~7.32]{Friz2010}. This leads to the \textit{Carnot-Caratheodory metric}~$d_{cc}$ via
\begin{equation*}
  d_{cc}(g,h):= \|g^{-1}\otimes h \|_{cc},\quad g,h\in  G^N(\mathbb{R}^n),
\end{equation*}
where  $g^{-1}$ is the inverse of $g$ in the sense $g^{-1}\otimes g=1$, see~\cite[Proposition~7.36 and Definition~7.41]{Friz2010}. Hence, $(G^N(\mathbb{R}^n),d_{cc})$  is a metric space.   

The space of all \textit{weakly geometric rough paths} of finite $q$-variation is then given by 
\begin{equation*}
  \Omega^q:=C^{q\var}([0,T];G^{\lfloor q \rfloor}(\mathbb{R}^n))
  := \bigg\{\X\in C([0,T];G^{\lfloor q \rfloor}(\mathbb{R}^n))\,:\, \|\X\|_{q\var} <\infty \bigg\},
\end{equation*}
where $\|\,\cdot\,\|_{q\var}$ is the $q$-variation with respect to the metric space~$(G^{\lfloor q \rfloor}(\mathbb{R}^n),d_{cc})$ as defined in~\eqref{eq:p-varition} and $\lfloor q \rfloor:=\max \{n\in \mathbb{N} \,:\, n\leq q\}$. Note that $\|\,\cdot\,\|_{q\var}$ on $\Omega^q$ is commonly called the homogeneous (rough path) norm since it is homogeneous with respect to the dilation map on $T^{\lfloor q \rfloor}(\R^n)$, cf.~\cite[Definition~7.13]{Friz2010}.\medskip

Coming back to the controlled differential equation~\eqref{eq:differential equation}, we first need to introduce a solution concept suitable for this equation given the driving signal is now a weakly geometric rough path. Let $V\colon \mathbb{R}^m \to \mathcal{L}(\mathbb{R}^n,\mathbb{R}^m)$ be a sufficiently smooth vector field and $y_0 \in \mathbb{R}^m$ be some initial condition. For a weakly geometric rough path $\X\in C^{q\var}([0,T];G^{\lfloor q \rfloor}(\mathbb{R}^n))$ we call $Y\in C([0,T];\mathbb{R}^m)$ a solution to the controlled differential equation (also called rough differential equation) 
\begin{equation}\label{eq:rough differential equation}
  \d Y_t = V(Y_t)\dd \X_t,\quad Y_0=y_0, \quad t\in [0,T],
\end{equation}
if there exist a sequence $(X^n)\subset C^{1\var}([0,T];\mathbb{R}^n)$ such that 
\begin{equation*}
  \lim_{n\to \infty} \sup_{0\leq s<t\leq T} d_{cc}(S_{\lfloor q\rfloor}(X^n)_{s,t},\X_{s,t})=0, \quad \sup_n \|S_{\lfloor q \rfloor}(X^n)\|_{q\var;[0,T]}<\infty,
\end{equation*}
and the corresponding solutions $Y^n$ to equation~\eqref{eq:rde integral} converge uniformly on $[0,T]$ to $Y$ as $n$ tends to $\infty$, cf.~\cite[Definition~10.17]{Friz2010}.\medskip

On the space $\Omega^q=C^{q\var}([0,T];G^{N}(\mathbb{R}^n))$ the classical way to restore to the continuity of the solution map associated to a controlled differential equation~\eqref{eq:rough differential equation} (also called It\^o-Lyons map) is to introduce the \textit{inhomogeneous variation distance} 
\begin{equation*}
  \rho_{q\var}(\X^1,\X^2):= \max_{k=1,\dots,N}\rho^{(k)}_{q\var;[0,T]}(\X^1,\X^2),
\end{equation*}
for $\X^1,\X^2 \in C^{q\var}([0,T];G^{N}(\mathbb{R}^n))$ and $q\in [1,\infty)$, with
\begin{equation*}
  \rho^{(k)}_{q\var;[s,t]}(\X^1,\X^2):= \sup_{\mathcal{P}\subset [s,t]} \bigg (\sum_{[u,v] \in \mathcal{P}} |\pi_k(\X_{u,v}^1-\X_{u,v}^2)|^{\frac{q}{k}} \bigg)^{\frac{k}{q}}, \quad [s,t]\subset [0,T],
\end{equation*}
for $k=1,\dots,N$, where $\pi_k\colon T^N(\mathbb{R}^n)\to \big(\mathbb{R}^n\big)^{\otimes k}$ is the projection to the $k$-th tensor level and each tensor level $\big(\mathbb{R}^n\big)^{\otimes k}$ is equipped with the Euclidean structure. Here we recall that $\X_{u,v}:=\X^{-1}_u \otimes \X_v$. Note that the distance~$ \rho_{q\var}$ is not homogeneous anymore with respect to the dilation map on $T^{N}(\R^n)$ as indicated by its name.

In the seminal paper~\cite{Lyons1998}, Lyons showed that the solution map $\Phi$ given by
\begin{equation*}
  \Phi \colon  \mathbb{R}^m \times \Lip^{\gamma} \times C^{1/\delta\var}([0,T];G^{\lfloor 1/\delta \rfloor}(\mathbb{R}^n)) \to C^{1/\delta\var}([0,T];\mathbb{R}^m) \quad \text{via} \quad \Phi(y_0,V,\X):=Y,
\end{equation*}
where $Y$ denotes the solution to equation~\eqref{eq:rough differential equation} given the input $(y_0,V,\X)$, is local Lipschitz continuity  with respect to the inhomogeneous variation distance~$\rho_{1/\delta\var}$ for any finite regularity $1/\delta >1$. \medskip

In the spirit of our characterization~\eqref{eq:holder-variation mixed norm} of the Riesz type variation we introduce for $\delta \in (0,1)$ and $p\in [1/\delta,\infty)$ the inhomogeneous \textit{mixed H\"older-variation distance} by
\begin{equation*}
  \rho_{\tilde V^{\delta,p}}(\X^1,\X^2):= \max_{k=1,\dots,N}\rho^{(k)}_{\tilde V^{\delta,p};[0,T]}(\X^1,\X^2),
\end{equation*}
for $\X^1,\X^2 \in \tilde V^{\delta,p}([0,T];G^{N}(\mathbb{R}^n))$ and for $k=1,\dots,N$, we set 
\begin{equation*}
  \rho^{(k)}_{\tilde V^{\delta,p};[s,t]}(\X^1,\X^2):= \sup_{\mathcal{P}\subset [s,t]} \bigg (\sum_{[u,v] \in \mathcal{P}} \frac{\rho^{(k)}_{1/\delta\var;[u,v]}(\X^1,\X^2)^{\frac{p}{k}}}{|u-v|^{\delta p-1}} \bigg)^{\frac{k}{p}}, \quad [s,t]\subset [0,T].
\end{equation*}
Furthermore, we define the \textit{Riesz type geometric rough path space} by
\begin{align*}
 \Omega^{\delta,p} 
 := V^{\delta,p}([0,T];G^{\lfloor 1/\delta \rfloor}(\mathbb{R}^n))
 =\tilde V^{\delta,p}([0,T];G^{\lfloor 1/\delta \rfloor}(\mathbb{R}^n))
 = \hat N^{\delta,p}([0,T];G^{\lfloor 1/\delta \rfloor}(\mathbb{R}^n)),
\end{align*}
for $\delta \in (0,1),\,p\in (1/\delta,\infty)$. The identifies hold due to Theorem~\ref{thm:riesz characterization} since $G^{\lfloor 1/\delta \rfloor}(\mathbb{R}^n)$ is a metric space equipped with the Carnot-Caratheodory metric~$d_{cc}$.\medskip

Relying on the equivalent characterization of Riesz type variation in terms of $q$-variation (cf. Theorem~\ref{thm:riesz characterization}) and on the inhomogeneous mixed H\"older-variation distances for Riesz type geometric rough paths, the local Lipschitz continuity of the It\^o-Lyons map $\Phi$ given by
\begin{equation}\label{eq:ito lyons map}
  \Phi \colon \mathbb{R}^m \times \Lip^{\gamma} \times \tilde V^{\delta,p}([0,T];G^{\lfloor 1/\delta \rfloor}(\mathbb{R}^n)) \to \tilde V^{\delta,p}([0,T];\mathbb{R}^m) \quad \text{via} \quad \Phi(y_0,V,\X):=Y,
\end{equation}
where $Y$ denotes the solution to equation~\eqref{eq:rough differential equation} given the input $(y_0,V,\X)$, can be obtained with respect to the inhomogeneous mixed H\"older-variation distance. 

\begin{theorem}\label{thm:ito map riesz variation}
  Let $\delta \in (0,1)$ and $\gamma, p \in (1,\infty)$ be such that $\delta > 1/p$ and $\gamma >1/\delta$. For a 
  Riesz type geometric rough path $\X\in \tilde V^{\delta,p}([0,T];G^{\lfloor 1/\delta\rfloor }(\mathbb{R}^n))$, for $V\in \Lip^{\gamma}$ and for every initial condition $y_0 \in \mathbb{R}^m$, the controlled differential equation~\eqref{eq:rough differential equation} has a unique solution $Y\in \tilde V^{\delta,p}([0,T];\mathbb{R}^m)$.
  
  Furthermore, the It\^o-Lyons map $\Phi$ as defined in \eqref{eq:ito lyons map} is locally Lipschitz continuous, that is, for $y_0^i \in \mathbb{R}^m$, $X^i\in \tilde V^{\delta,p}([0,T];G^{\lfloor 1/\delta\rfloor}(\mathbb{R}^n))$ and $V^i\in \Lip^{\gamma}$ satisfying 
  \begin{equation*}
    \|\X^i\|_{ \tilde V^{\delta,p}}\leq b \quad \text{and} \quad \|V^i\|_{\Lip^\gamma}\leq l, \quad i=1,2,
  \end{equation*}
  for some $b,l>0$, with corresponding solution $Y^i$, there exist a constant $C=C(b,l,\gamma,\delta,p)\geq 1$ such that 
  \begin{equation*}
    \|Y^1-Y^2\|_{ \tilde V^{\delta,p}} \leq C \big(\|V^1-V^2\|_{\Lip^{\gamma-1}} + |y_0^1-y^2_0|+ \rho_{ \tilde V^{\delta,p}}(\X^1,\X^2)\big).
  \end{equation*}
\end{theorem}

Before we come to the proof, we recall that a continuous function $\omega\colon \Delta_T \to [0,\infty)$ is called \textit{control function} if $\omega(s,s)=0$ for $s\in[0,T]$ and $\omega$ is super-additive, cf. Remark~\ref{rmk:supper additive}.

\begin{proof}
  Due to Proposition~\ref{prop:variation embeddings} and Theorem~\ref{thm:riesz characterization}, the assumption $\X\in \tilde V^{\delta,p}([0,T];G^{\lfloor 1/\delta\rfloor }(\mathbb{R}^n))$ implies $\X\in C^{1/\delta\var}([0,T];G^{\lfloor 1/\delta\rfloor }(\mathbb{R}^n))$ as $\delta > 1/p$. Therefore, the controlled differential equation~\eqref{eq:rough differential equation} has a unique solution $Y\in C^{1/\delta\var}([0,T];\mathbb{R}^m)$ given the regularity of the vector field $V\in \Lip^{\gamma}$ with $\gamma >1/\delta$, see \cite[Theorem~10.26]{Friz2010}.
  
  It remains to show the Riesz type variation of $Y$ and the local Lipschitz continuity of the It\^o-Lyons map~$\Phi$ as defined in~\eqref{eq:ito lyons map}. For this purpose we choose a suitable control function $\omega$ (see \eqref{eq:control function} for the specific definition) and introduce the distance
  \begin{equation*}
    \rho_{1/\delta\om} (\X^1,\X^2):=\max_{k=1,\dots, \lfloor 1/\delta\rfloor}\rho_{1/\delta\om;[0,T]}^{(k)} (\X^1,\X^2)
  \end{equation*}
  for $\X^1, \X^2 \in \tilde V^{\delta,p}([0,T];G^{\lfloor 1/\delta\rfloor }(\mathbb{R}^n))$ and 
  \begin{equation*}
    \rho_{1/\delta\om;[0,T]}^{(k)} (\X^1,\X^2):=\sup_{0\leq s<t\leq T} \frac{|\pi_k(\X^1_{s,t}-\X^2_{s,t})|}{\omega(s,t)^{k\delta }},\quad k=1,\dots, \lfloor 1/\delta\rfloor.
  \end{equation*}
  As one can see in the proof of \cite[Theorem~10.26]{Friz2010}, one has the following two estimates for the control function $\omega$ and a constant $C=C(\gamma,\delta)>0$:
  \begin{align}\label{eq:initial continuous}
    \|Y^1-Y^2&\|_{\infty;[0,T]}\nonumber\\
    &\leq C \bigg (|y_0^1-y_0^2| + \frac{1}{l} \|V^1-V^2\|_{\Lip^{\gamma-1}}+\rho_{1/\delta\om} (\X^1,\X^2) \bigg)\exp (Cl^{1/\delta}\omega (0,T))
  \end{align}
  and, for all $u<v$ in $[0,T]$,
  \begin{align}\label{eq:increment}
    |Y^1_{u,v}&-Y^2_{u,v}| \nonumber\\
    &\leq C  \bigg (l |Y_u^1-Y_u^2| +  \|V^1-V^2\|_{\Lip^{\gamma-1}}+l\rho_{1/\delta\om} (\X^1,\X^2) \bigg)\omega(u,v)^{\delta}\exp (Cl^{1/\delta}\omega (0,T)).
  \end{align}
  From inequality~\eqref{eq:increment} we deduce that 
  \begin{align*}
    \|& Y^1-Y^2\|_{1/\delta\var;[s,t]}^{\frac{1}{\delta}}= \sup_{\mathcal{P}\subset [s,t]}\sum_{[u,v]\in \mathcal{P}} |Y^1_{u,v}-Y^2_{u,v}|^{1/\delta} \\
    & \lesssim  \bigg (l \|Y^1-Y^2\|_{\infty;[s,t]} +  \|V^1-V^2\|_{\Lip^{\gamma-1}}+l\rho_{1/\delta\om} (\X^1,\X^2) \bigg)^{1/\delta}\omega(s,t)\exp (Cl^{1/\delta}\omega (0,T))^{1/\delta},
  \end{align*}
  which further leads to 
  \begin{align*}
    \|Y^1-Y^2\|_{\tilde V^{\delta,p}}^{p}=& \sup_{\mathcal{P}\subset [0,T]}\sum_{[s,t]\in \mathcal{P}} \frac{\| Y^1-Y^2\|_{1/\delta\var;[s,t]}^p}{|t-s|^{\delta p -1 }} \\
    \lesssim & \bigg (l \|y^1-y^2\|_{\infty;[0,T]} +  \|V^1-V^2\|_{\Lip^{\gamma-1}}+l\rho_{1/\delta\om} (\X^1,\X^2) \bigg)^{p} \\
    &\times \exp (Cl^{1/\delta}\omega (0,T))^{p} \bigg(\sup_{\mathcal{P}\subset [0,T]}\sum_{[s,t]\in \mathcal{P}} \frac{\omega(s,t)^{\delta p}}{|t-s|^{\delta p -1 }}\bigg)^p.
  \end{align*}
  Plugging in estimate~\eqref{eq:initial continuous} in the last inequality gives 
  \begin{align}\label{eq:estimate ito map}
  \begin{split}
    \|Y^1-Y^2\|_{\tilde V^{\delta,p}}
    \lesssim & \bigg (l |y^1_0-y_0^2| +  \|V^1-V^2\|_{\Lip^{\gamma-1}}+l\rho_{1/\delta\om} (\X^1,\X^2) \bigg) \\
    &\times \exp (Cl^{1/\delta}\omega (0,T)) \bigg(\sup_{\mathcal{P}\subset [0,T]}\sum_{[s,t]\in \mathcal{P}} \frac{\omega(s,t)^{\delta p}}{|t-s|^{\delta p-1 }}\bigg).
  \end{split}
  \end{align}
  In order to complete the proof, we consider the control function 
  \begin{equation}\label{eq:control function}
    \omega (s,t):= \|\X^1\|_{1/\delta \var; [s,t]}^{\frac{1}{\delta}}+\|\X^2\|_{1/\delta \var; [s,t]}^{\frac{1}{\delta}}+\sum_{k=1}^{\lfloor 1/\delta \rfloor} \omega_{\X^1,\X^2}^{(k)}(s,t),\quad (s,t)\in \Delta_T,
  \end{equation}
  where 
  \begin{equation*}
    \omega_{\X^1,\X^2}^{(k)}(s,t):= \bigg (\frac{\rho_{1/\delta\var;[s,t]}^{(k)} (\X^1,\X^2)}{\rho_{\tilde V^{\delta,p};[0,T]}^{(k)} (\X^1,\X^2)}\bigg)^{\frac{1}{\delta k}}
  \end{equation*}
  with the convention $0/0:=0$, and investigate some properties of $\omega$. First notice that $\omega$ fulfills all the requirements to apply \cite[Theorem~10.26]{Friz2010}.
  Moreover, it is easy to see that 
  \begin{equation}\label{eq:estimate control function 1}
    \rho_{1/\delta \om} (\X^1,\X^2)\lesssim \rho_{\tilde V^{\delta,p}} (\X^1,\X^2)
  \end{equation}
  as one has for $k=1,\dots,\lfloor 1/\delta \rfloor$ and $0\leq s < t\leq T$ the following estimate 
  \begin{align*}
    |\pi_k (\X^1_{s,t}-\X^2_{s,t})| \leq \frac{\rho_{1/\delta \om;[s,t]}^{(k)}(\X^1,\X^2)}{\rho_{\tilde V^{\delta,p};[0,T]}^{(k)} (\X^1,\X^2)}\rho_{\tilde V^{\delta,p};[0,T]}^{(k)}(\X^1,\X^2)
    \leq \omega(s,t)^{\delta k} \rho_{\tilde V^{\delta,p}}(\X^1,\X^2).
  \end{align*}
  The last observation on $\omega$ we need is 
  \begin{equation}\label{eq:estimate control function 2}
    \sup_{\mathcal{P}\subset [0,T]}\sum_{[s,t]\in \mathcal{P}} \frac{\omega(s,t)^{\delta p}}{|t-s|^{\delta p-1 }}\lesssim \|\X^1\|_{\tilde V^{\delta,p}}^p +\|\X^2\|_{\tilde V^{\delta,p}}^p+1.
  \end{equation}
  Indeed, by Proposition~\ref{prop:variation embeddings} we have for every partition $\mathcal{P}$ of $[0,T]$ 
  \begin{equation*}
    \sum_{[s,t]\in \mathcal{P}}\frac{\|\X^i\|_{1/\delta \var;[s,t]}^p}{|t-s|^{\delta p-1}}
    \lesssim \sum_{[s,t]\in \mathcal{P}}\frac{\|\X^i\|_{\tilde V^{\delta,p};[s,t]}^p|t-s|^{\delta p-1}}{|t-s|^{\delta p-1}}\lesssim \|\X^i\|^p_{\tilde V^{\delta,p}},\quad i=1,2,
  \end{equation*}
  and further using 
  \begin{align*}
    \rho^{(k)}_{1/\delta \var;[s,t]}(\X^1,\X^2)
    &\leq \bigg( \frac{\rho^{(k)}_{1/\delta \var;[s,t]}(\X^1,\X^2)^{\frac{p}{k}} }{|t-s|^{\delta p -1} }\bigg)^{\frac{k}{p}} |t-s|^{(\delta-1/p)k}\\
    &\leq \rho^{(k)}_{\tilde V^{\delta,p};[s,t]}(\X^1,\X^2)  |t-s|^{(\delta-1/p)k},\quad \text{for }  k=1,\dots,\lfloor 1/\delta \rfloor,
  \end{align*}
  we arrive at
  \begin{align*}
    \sum_{[s,t]\in \mathcal{P}} \frac{\omega^{(k)}_{\X^1,\X^2}(s,t)^{\delta p} }{|t-s|^{\delta p-1}} 
    &= \rho^{(k)}_{\tilde V^{\delta,p};[0,T]}(\X^1,\X^2)^{-\frac{p}{k}}\sum_{[s,t]\in \mathcal{P}} \frac{ \rho^{(k)}_{1/\delta\var;[s,t]}(\X^1,\X^2)^{\frac{p}{k}} }{|t-s|^{\delta p-1}} \\
    & \leq\rho^{(k)}_{\tilde V^{\delta,p};[0,T]}(\X^1,\X^2)^{-\frac{p}{k}}  \sum_{[s,t]\in \mathcal{P}} \rho^{(k)}_{\tilde V^{\delta,p};[s,t]}(\X^1,\X^2)^{\frac{p}{k}} 
    \leq 1.
  \end{align*}
  By combining these estimates we deduce~\eqref{eq:estimate control function 2}.
  
  Therefore, estimate~\eqref{eq:estimate ito map} together with \eqref{eq:estimate control function 1} and \eqref{eq:estimate control function 2} reveals 
  \begin{align*}
    \|Y^1-Y^2\|_{\tilde V^{\delta,p}} 
    \lesssim & \bigg (l |y^1_0-y_0^2| +  \|V^1-V^2\|_{\Lip^{\gamma-1}}+l\rho_{\tilde V^{\delta,p}} (\X^1,\X^2) \bigg)\\
    &\times \exp (Cl^{1/\delta}(2b+1)) (2b+1),
  \end{align*}
  which completes the proof.
\end{proof}

\subsection{Inhomogeneous Riesz type distance}

In the context of rough path theory it is very convenient to work with the characterization of Riesz type variation in terms of classical $q$-variation and to introduce the corresponding inhomogeneous mixed H\"older-variation distance $\rho_{\tilde V^{\delta,p}}$, as we have seen in the previous subsection. 
However, also the other other characterizations of Riesz type variation allow for obtaining the local Lipschitz continuity of the It\^o-Lyons map $\Phi$ as defined in~\eqref{eq:ito lyons map}. \medskip

Keeping in mind the Riesz type variation~\eqref{eq:Riesz variation}, we define for $\delta \in (0,1)$ and $p\in [1/\delta,\infty)$ inhomogeneous \textit{Riesz type distance} by
\begin{equation*}
  \rho_{V^{\delta,p}}(\X^1,\X^2):= \max_{k=1,\dots,\lfloor 1/\delta \rfloor}\rho^{(k)}_{ V^{\delta,p};[0,T]}(\X^1,\X^2),
\end{equation*}
for $\X^1,\X^2 \in V^{\delta,p}([0,T];G^{\lfloor 1/\delta \rfloor}(\mathbb{R}^n))$, where we set 
\begin{equation*}
  \rho^{(k)}_{V^{\delta,p};[s,t]}(\X^1,\X^2)
  := \sup_{\mathcal{P}\subset [s,t]} \bigg (\sum_{[u,v] \in \mathcal{P}} \frac{ |\pi_k(\X^1_{u,v}-\X^2_{u,v})|^{\frac{p}{k}}}{|u-v|^{\delta p-1}} \bigg)^{\frac{k}{p}}, \quad [s,t]\subset [0,T],
\end{equation*}
for $k=1,\dots,\lfloor 1/\delta \rfloor$. Indeed, the inhomogeneous Riesz type distance and inhomogeneous mixed H\"older-variation distance are equivalent.

\begin{lemma}\label{lem:equivalence Riesz distance}
  If $\delta \in (0,1)$ and $p \in (1,\infty)$ with $\delta > 1/p$, then 
  \begin{equation*}
    \rho_{V^{\delta,p}}(\X^1,\X^2) \lesssim \rho_{\tilde V^{\delta,p}}(\X^1,\X^2)\lesssim \rho_{V^{\delta,p}}(\X^1,\X^2)
  \end{equation*}
  for $\X^1,\X^2 \in V^{\delta,p}([0,T];G^{\lfloor 1/\delta \rfloor}(\mathbb{R}^n))$.
\end{lemma}

\begin{proof}
  The first inequality directly follows from
  \begin{equation*}
    |\pi_k (\X^1_{u,v}-\X^2_{u,v})|^{\frac{p}{k}}\leq \rho^{(k)}_{1/\delta\var;[u,v]} (\X^1,\X^2)^{\frac{p}{k}},\quad [u,v]\subset [0,T],
  \end{equation*}
  for $k=1,\dots,\lfloor 1/\delta \rfloor$.
  
  For the second inequality we notice 
  \begin{align*}
    |\pi_k (\X^1_{u,v}-\X^2_{u,v})|^{\frac{1}{\delta k}}
    &\leq \bigg(\frac{|\pi_k (\X^1_{u,v}-\X^2_{u,v})|^{\frac{p}{k}}}{|u-v|^{\delta p -1}}\bigg)^{\frac{1}{\delta p}}|u-v|^{1-\frac{1}{\delta p}}\\
    &\leq \rho_{V^{\delta,p};[u,v]}^{(k)}(\X^1,\X^2)^{\frac{1}{\delta k}}|u-v|^{1-\frac{1}{\delta p}},\quad [u,v]\subset [0,T],
  \end{align*}
  for $k=1,\dots,\lfloor 1/\delta \rfloor$. Due to Remark~\ref{rmk:supper additive} the function
  \begin{equation*}
    \Delta_T\ni (u,v)\mapsto \rho_{V^{\delta,p};[u,v]}^{(k)}(\X^1,\X^2)^{\frac{1}{\delta k}}|u-v|^{1-\frac{1}{\delta p}}
  \end{equation*}
  is super-additive and thus we get
  \begin{equation*}
    \rho^{(k)}_{1/\delta \var;[u,v]}(\X^1,\X^2)^{\frac{p}{k}} \leq \rho^{(k)}_{V^{\delta,p};[u,v]}(\X^1,\X^2) |u-v|^{\delta p-1}.
  \end{equation*}
  Therefore, using the super-additive of $\rho^{(k)}_{V^{\delta,p};[u,v]}(\X^1,\X^1)^{\frac{p}{k}}$ as functions in $(u,v)\in \Delta_T$, we obtain 
  \begin{equation*}
    \rho^{(k)}_{\tilde V^{\delta ,p};[u,v]}(\X^1,\X^2)^{\frac{p}{k}} \leq \rho^{(k)}_{V^{\delta,p};[u,v]}(\X^1,\X^2)^\frac{p}{k}
  \end{equation*}
  for $k=1,\dots,\lfloor 1/\delta \rfloor$, which implies the second inequality.
\end{proof}

Applying the equivalence of the inhomogeneous distances $\rho_{V^{\delta, p}}$ and $\rho_{\tilde V{\delta, p}}$ (Lemma~\ref{lem:equivalence Riesz distance})  and the characterization of Riesz type spaces (Theorem~\ref{thm:riesz characterization}), the local Lipschitz continuity of the It\^o-Lyons map~\eqref{eq:ito lyons map} with respect to $\rho_{V^{\delta, p}}$ is an immediate consequence of Theorem~\ref{thm:ito map riesz variation}:

\begin{corollary}\label{cor:ito map riesz type norm}
  Let $\delta \in (0,1)$ and $\gamma, p \in (1,\infty)$ be such that $\delta > 1/p$ and $\gamma >1/\delta$. 
  
  The It\^o-Lyons map  
  \begin{equation*}
    \Phi \colon  \mathbb{R}^m \times \Lip^{\gamma} \times V^{\delta,p}([0,T];G^{\lfloor 1/\delta \rfloor}(\mathbb{R}^n)) \to V^{\delta,p}([0,T];\mathbb{R}^m) \quad \text{via} \quad \Phi(y_0,V,\X):=Y,
  \end{equation*}
  where $Y$ denotes the solution to controlled differential equation~\eqref{eq:rough differential equation} given the input $(y_0,V,\X)$, is locally Lipschitz continuous with respect to inhomogeneous Riesz type distance $\rho_{V^{\delta,p}}$.
\end{corollary}

\subsection{Inhomogeneous Nikolskii type distance}
 
To complete the picture, we provide in this subsection an inhomogeneous distance in terms of Nikolskii regularity (cf.~\eqref{eq:nikolskii norm integral}), which is locally Lipschitz equivalent to the inhomogeneous Riesz type distance and which ensures the local Lipschitz continuity of the It\^o-Lyons map $\Phi$ as defined in~\eqref{eq:ito lyons map}. \medskip

To that end we introduce the inhomogeneous Nikolskii type distance as follows: For $\X^1,\X^2 \in \hat N^{\delta, p}([0,T];G^{\lfloor 1/\delta \rfloor}(\mathbb{R}^n))$ we set 
\begin{equation*}
  \rho_{N^{\delta,p};[u,v]}^{(k)}(\X^1,\X^2):=\sup_{|v-u|\geq h>0} h^{-\delta k}\bigg ( \int_u^{v-h} |\pi_k (\X^1_{r,r+h}-\X^2_{r,r+h})|^{\frac{p}{k}}\dd r \bigg)^{\frac{k}{p}},\quad [u,v]\subset [0,T],
\end{equation*}
and 
\begin{equation*}
  \rho_{\hat N^{\delta,p};[s,t]}^{(k)}(\X^1,\X^2):=\sup_{\mathcal{P}\subset [s,t]} \bigg(\sum_{[u,v]\in\mathcal{P}} \rho_{N^{\delta,p};[u,v]}^{(k)}(\X^1,\X^2)^\frac{p}{k}\bigg)^{\frac{k}{p}},\quad [s,t]\subset [0,T],
\end{equation*}
for $k=1,\dots,\lfloor 1/\delta \rfloor$. The \textit{inhomogeneous Nikolskii type distance} $\rho_{\hat N^{\delta,p}}$ is defined by
\begin{equation*}
  \rho_{\hat N^{\delta,p}}(\X^1,\X^2):= \max_{k=1,\dots,N}\rho^{(k)}_{\hat N^{\delta,p};[0,T]}(\X^1,\X^2).
\end{equation*}

In the next two lemmas (Lemma~\ref{lem:nikolskii distance 1} and~\ref{lem:nikolskii distance 2}) we establish that the two ways of introducing an inhomogeneous distance on $C([0,T];G^{\lfloor 1/\delta \rfloor}(\mathbb{R}^n))$, given by the inhomogeneous mixed H\"older-variation distance and the inhomogeneous Nikolskii type distance, are locally equivalent. 

\begin{lemma}\label{lem:nikolskii distance 1}
  If $\delta \in (0,1)$ and $p \in (1,\infty)$ with $\delta > 1/p$, then
  \begin{equation*}
    \rho_{\hat N^{\delta,p}}(\X^1,\X^2) \lesssim \rho_{\tilde V^{\delta,p}}(\X^1,\X^2)
  \end{equation*}
  for $\X^1,\X^2 \in \tilde V^{\delta,p}([0,T];G^{\lfloor 1/\delta \rfloor}(\mathbb{R}^n))$.
\end{lemma}

\begin{proof}
  Let $\X^1,\X^2 \in \tilde V^{\delta,p}([0,T];G^{\lfloor 1/\delta \rfloor}(\mathbb{R}^n))$ and $k=1,\dots,N$. For $[s,t]\subset [0,T]$ and $h\in (0,t-s]$ fixed we consider the dissection $\mathcal{P}(h):=\{[t_i,t_{i+1}]\,:\, s=t_0 < \dots < t_M=t-h\}$ such that 
  \begin{equation*}
    |t_M -t_{M-1}|\leq h\quad \text{and} \quad |t_{i+1}-t_i | = h \quad\text{for } i=0,\dots,M-2,\quad M\in \mathbb{N}.
  \end{equation*}
  Since $\sup_{u\in [t_i,t_{i+1}]}|\pi_k (\X^1_{u,u+h}-\X^2_{u,u+h})|\leq \rho_{1/\delta\var;[t_i,t_{i+2}]}(\X^1,\X^2)$ for $i=1,\dots, M-1$ with $t_{M+1}:=t-h$, we deduce that 
  \begin{align*}
    \int_s^{t-h}|\pi_k (\X^1_{u,u+h}-\X^2_{u,u+h})|^{\frac{p}{k}}\dd u  &\leq \sum_{i=1}^{M-1} \sup_{u\in [t_i,t_{i+1}]}|\pi_k (\X^1_{u,u+h}-\X^2_{u,u+h})|^{\frac{p}{k}} h\\
    &\leq \frac{1}{2} (2h)^{\delta p} \sum_{i=1}^{M-1} \frac{\rho_{1/\delta\var;[t_i,t_{i+2}]}(\X^1,\X^2)^{\frac{p}{k}}}{|t_i-t_{i+2}|^{\delta p-1}}\\ 
    &\lesssim h^{\delta p}\rho_{\tilde V^{\delta,p};[s,t]}^{(k)}(\X^1,\X^2)^{\frac{p}{k}}.
  \end{align*}
  Therefore, by the super-additivity of $\rho_{\hat N^{\delta,p};[s,t]}^{(k)}(\X^1,\X^2)^{\frac{p}{k}}$ and $\rho_{\tilde V^{\delta,p};[s,t]}^{(k)}(\X^1,\X^2)^{\frac{p}{k}}$ as functions in $(s,t)\in \Delta_T$ we obtain 
  \begin{equation*}
    \rho_{\hat N^{\delta,p};[0,T]}^{(k)}(\X^1,\X^2) \lesssim \rho_{\tilde V^{\delta,p};[0,T]}^{(k)}(\X^1,\X^2)
  \end{equation*}
  for every $k=1,\dots,\lfloor 1/\delta \rfloor$, which implies that $\rho_{\hat N^{\delta,p}}(\X^1,\X^2) \lesssim \rho_{\tilde V^{\delta,p}}(\X^1,\X^2)$.
\end{proof}

\begin{lemma}\label{lem:nikolskii distance 2}
  Let $\delta \in (0,1)$ and $p \in (1,\infty)$ with $\delta > 1/p$.  For Riesz type geometric rough paths $\X^1,\X^2 \in \hat N^{\delta,p}([0,T];G^{\lfloor 1/\delta \rfloor}(\mathbb{R}^n))$ there exists a constant 
  \begin{equation*}
    C:=C\big(\delta, p, \|\X^1\|_{\hat N^{\delta,p}},\|\X^2\|_{\hat N^{\delta,p}}\big)\geq 1,
  \end{equation*}
  depending only on $\delta$, $p$ and the upper bound of $\|\X^1\|_{\hat N^{\delta,p}}$ and $\|\X^2\|_{\hat N^{\delta,p}}$,
  such that 
  \begin{equation*}
    \rho_{\tilde V^{\delta,p}}(\X^1,\X^2) \leq C \rho_{\hat N^{\delta,p}}(\X^1,\X^2).
  \end{equation*}
\end{lemma}

\begin{proof}
  Let $\X^1,\X^2 \in \hat N^{\delta,p}([0,T];G^{\lfloor 1/\delta \rfloor}(\mathbb{R}^n))$ be Riesz type geometric rough paths. In order to prove the inequality in Lemma~\ref{lem:nikolskii distance 2}, it is sufficient to show that there exists a constant $C=C\big(\delta, p, \|\X^1\|_{\hat N^{\delta,p}},\|\X^2\|_{\hat N^{\delta,p}}\big)\geq 1$ such that
  \begin{equation}\label{eq:control each level}
    |\pi_j (\X^1_{s,t}-\X^2_{s,t})|^{\frac{1}{\delta j}} \leq C \bigg( \sum_{i=1}^j \rho_{\hat N^{\delta,p};[s,t]}^{(i)} (\X^1,\X^2)^{\frac{p}{j}}\bigg )^{\frac{1}{\delta p}} |t-s|^{1-\frac{1}{\delta p}}=: \omega^{(j)}(s,t)
  \end{equation}
  for all $s,t\in [0,T]$ with $s<t$ and for every $j=1,\dots,\lfloor 1/\delta \rfloor$.
  
  Indeed, for each $j=1,\dots,\lfloor 1/\delta \rfloor$ the function $\omega^{(j)}\colon \Delta_T\to [0,\infty)$ is the super-additive:
  \begin{align*}
    \omega^{(j)}(s,t)+& \omega^{(j)}(t,u)\\
    &\leq C\bigg( \sum_{i=1}^j \rho_{\hat N^{\delta,p};[s,t]}^{(i)} (\X^1,\X^2)^{\frac{p}{j}}+\rho_{\hat N^{\delta,p};[u,t]}^{(i)} (\X^1,\X^2)^{\frac{p}{j}}\bigg )^{\frac{1}{\delta p}} \big( |t-s| + |u-t|\big)^{1-\frac{1}{\delta p}}\\
    &\leq C \bigg( \sum_{i=1}^j \rho_{\hat N^{\delta,p};[s,u]}^{(i)} (\X^1,\X^2)^{\frac{p}{j}}\bigg )^{\frac{1}{\delta p}}|u-s|^{1-\frac{1}{\delta p}}
    = \omega^{(j)}(s,u), 
  \end{align*}
  for $0\leq s\leq t\leq u\leq T$, where we used H\"older's inequality and $p/j\geq 1$. This implies 
  \begin{align*}
    \rho_{1/\delta\var;[s,t]}^{(j)}(\X^1,\X^2)
    &= \sup_{\mathcal{P}\subset [s,t]} \bigg( \sum_{[u,v]\in \mathcal{P}}|\pi_j(\X^1_{u,v}-\X^2_{u,v})|^\frac{1}{\delta j} \bigg)^{j\delta}\\
    &\leq C \bigg( \sum_{i=1}^j \rho_{\hat N^{\delta,p};[s,t]}^{(i)} (\X^1,\X^2)^{\frac{p}{j}}\bigg )^{\frac{j}{p}}|t-s|^{\frac{j}{p}(\delta p -1)},
  \end{align*}
  where the super-additivity of $\omega^{(j)}$ is applied in last line. Hence, we get further 
  \begin{align*}
    \rho^{(j)}_{\tilde V^{\delta,p};[0,T]}(\X^1,\X^2)
    &= \sup_{\mathcal{P}\subset [0,T]} \bigg( \sum_{[u,v]\in\mathcal{P}} \frac{\rho_{1/\delta\var;[u,v]}^{(j)}(\X^1,\X^2)^{\frac{p}{j}}}{|u-v|^{\delta p -1}} \bigg )^{\frac{j}{p}} \\
    &\leq C  \sup_{\mathcal{P}\subset [0,T]} \bigg( \sum_{[u,v]\in\mathcal{P}} \sum_{i=1}^j \rho_{\hat N^{\delta,p};[u,v]}^{(i)} (\X^1,\X^2)^{\frac{p}{j}} \bigg )^{\frac{j}{p}}
    \leq C \sum_{i=1}^j \rho_{\hat N^{\delta,p};[0,T]}^{(i)} (\X^1,\X^2),
  \end{align*}
  which immediately gives by taking the maximum over $j=1,\dots,\lfloor 1/\delta \rfloor$ that
  \begin{equation*}
    \rho_{\tilde V^{\delta,p}}(\X^1,\X^2) \leq C \rho_{\hat N^{\delta,p}}(\X^1,\X^2).
  \end{equation*}
  
  In order to prove inequality~\eqref{eq:control each level} for each $j=1,\dots,\lfloor 1/\delta \rfloor$, we argue via induction. For $j=1$ inequalities \eqref{eq:control each level} is an easy consequence of Theorem~\ref{thm:riesz characterization}. We now assume that \eqref{eq:control each level} is true for the levels $j=1,\dots,k-1$ and establish the inequality for level $k$.
  Let us fix $s,t\in [0,T]$ with $s<t$ and define 
  \begin{equation*}
    Z^s_u := \pi_k (\X^1_{s,s+u}-\X^2_{s,s+u}),\quad u\in [0,t-s].
  \end{equation*}
  For $u,h\in [0,t-s]$ with $u+h\in [0,t-s]$ and using  
  \begin{align*}
    \X^1_{s,s+u+h}-\X^1_{s,s+u}=\X^1_{s,s+u}\otimes (\X^1_{s+u,s+u+h}-1),\quad \pi_0(\X^1_{s+u,s+u+h}-1)=0 
  \end{align*}  
  and
  \begin{align*}
    \pi_j(\X^1_{s+u,s+u+h}-1)=\pi_j(\X^1_{s+u,s+u+h})\quad \text{ for } j>0,
  \end{align*}
  we obtain 
  \begin{align*}
    Z^s_{u+h}-Z^s_{u}=& \pi_k (\X^1_{s,s+u+h}-\X^1_{s,s+u})-\pi_k (\X^2_{s,s+u+h}-\X^2_{s,s+u}) \\
    = & \sum_{j=1}^{k} \pi_{k-j} (\X^1_{s,s+u})\otimes \pi_j(\X^1_{s+u,s+u+h}) - \sum_{j=1}^{k} \pi_{k-j} (\X^2_{s,s+u})\otimes \pi_j(\X^2_{s+u,s+u+h})\\
    = & \sum_{j=1}^{k} \pi_{k-j} (\X^1_{s,s+u})\otimes \pi_j(\X^1_{s+u,s+u+h}-\X^2_{s+u,s+u+h})\allowdisplaybreaks\\ 
      & + \sum_{j=1}^{k} \pi_{k-j} (\X^1_{s,s+u}-\X^2_{s,s+u})\otimes \pi_j(\X^2_{s+u,s+u+h}).
  \end{align*}
  Keeping in mind $\pi_0(\X^1_{s,s+u}-\X^2_{s,s+u})=0$, we arrive at
  \begin{align*}
    Z^s_{u+h}-Z^s_{u}
    = & \sum_{j=1}^{k-1} \pi_{k-j} (\X^1_{s,s+u})\otimes \pi_j(\X^1_{s+u,s+u+h}-\X^2_{s+u,s+u+h})\\
      & + \sum_{j=1}^{k-1} \pi_{k-j} (\X^1_{s,s+u}-\X^2_{s,s+u})\otimes \pi_j(\X^2_{s+u,s+u+h})\\
      & + \pi_k(\X^1_{s+u,s+u+h}-\X^2_{s+u,s+u+h}).
  \end{align*}
  Hence, we get the following estimate 
  \begin{equation*}
    \sup_{|t-s|\geq h>0} h^{-\delta}\bigg (\int_s^{t-h} |Z^s_{u+h}-Z^s_{u}|^p \dd u \bigg)^{\frac{1}{p}}\lesssim \Delta_1 + \Delta_2 + \Delta_3 
  \end{equation*}
  where we set 
  \begin{align*}
    & \Delta_1 := \sum_{j=1}^{k-1} \sup_{|t-s|\geq h>0} h^{-\delta}\bigg (\int_s^{t-h} \|\X^1_{s,s+u}\|^{p(k-j)} |\pi_j(\X^1_{s+u,s+u+h}-\X^2_{s+u,s+u+h})|^p  \dd u \bigg)^{\frac{1}{p}},\\
    & \Delta_2 := \sum_{j=1}^{k-1} \sup_{|t-s|\geq h>0} h^{-\delta}\bigg (\int_s^{t-h} |\pi_{k-j} (\X^1_{s,s+u}-\X^2_{s,s+u})|^p \|\X^2_{s+u,s+u+h}\|^{pj} \dd u \bigg)^{\frac{1}{p}},\allowdisplaybreaks\\
    & \Delta_3 := \sup_{|t-s|\geq h>0} h^{-\delta}\bigg (\int_s^{t-h} |\pi_k(\X^1_{s+u,s+u+h}-\X^2_{s+u,s+u+h})|^p\dd u \bigg)^{\frac{1}{p}}.
  \end{align*}
  Due to Proposition~\ref{prop:variation embeddings} and $\delta >1/p$, we have 
  \begin{equation*}
    \|\X^1_{s,s+u}\|^{p(k-j)}\leq \|\X^{1}\|^{p(k-j)}_{1/\delta\var;[s,t]} \lesssim  \|\X^{1}\|^{p(k-j)}_{\hat N^{\delta,p};[s,t]} |t-s|^{(\delta -\frac{1}{p})p(k-j)}.
  \end{equation*}
  Moreover, the induction hypothesis gives 
  \begin{equation*}
    |\pi_j (\X^1_{s+u,s+u+h}-\X^2_{s+u,s+u+h})|^{(1-\frac{1}{j})p} \lesssim \bigg(\sum_{i=1}^j \rho_{\hat N^{\delta,p};[s,t]}^{(i)} (\X^1,\X^2)^{\frac{p}{j}}\bigg )^{j-1} |t-s|^{(j-1)(\delta-\frac{1}{p})p}.
  \end{equation*}
  Therefore, $\Delta_1$ can be estimated by
  \begin{align*}
    \Delta_1 \lesssim &  \sum_{j=1}^{k-1} \|\X^1\|^{k-j}_{\hat N^{\delta,p};[s,t]} |t-s|^{(k-j)(\delta-\frac{1}{p})} \bigg(\sum_{i=1}^j \rho_{\hat N^{\delta,p};[s,t]}^{(i)} (\X^1,\X^2)^{\frac{p}{j}}\bigg )^{\frac{j-1}{p}} |t-s|^{(j-1)(\delta -\frac{1}{p})} \\
    & \times \sup_{|t-s|\geq h>0} h^{-\delta}\bigg (\int_s^{t-h} |\pi_j(\X^1_{s+u,s+u+h}-\X^2_{s+u,s+u+h})|^{\frac{p}{j}} \dd u \bigg)^{\frac{1}{p}}\allowdisplaybreaks \\
    \leq &  \sum_{j=1}^{k-1} \|\X^1\|^{k-j}_{\hat N^{\delta,p};[s,t]} \rho_{\hat N^{\delta,p};[s,t]}^{(j)} (\X^1,\X^2)^{\frac{1}{j}} \bigg(\sum_{i=1}^j \rho_{\hat N^{\delta,p};[s,t]}^{(i)} (\X^1,\X^2)^{\frac{p}{j}}\bigg )^{\frac{j-1}{p}} |t-s|^{(k-1)(\delta-\frac{1}{p})}\\
    \leq & \sum_{j=1}^{k-1} \|\X^1\|^{k-j}_{\hat N^{\delta,p};[s,t]} \bigg(\sum_{i=1}^j \rho_{\hat N^{\delta,p};[s,t]}^{(i)} (\X^1,\X^2)^{\frac{p}{j}}\bigg )^{\frac{j}{p}} |t-s|^{(\delta-\frac{1}{p})(k-1)}.
  \end{align*}
  For $\Delta_2$ we first observe again due to Proposition~\ref{prop:variation embeddings} that
  \begin{align*}
    \|\X^2_{s+u,s+u+h}\|^{pj}\lesssim \|\X^2_{s+u,s+u+h}\|^{p} \|\X^2\|^{p(j-1)}_{\hat N^{\delta,p};[s,t]} |t-s|^{(\delta -\frac{1}{p})(j-1)p}
  \end{align*}
  and by the induction hypothesis that 
  \begin{align*}
    |\pi_{k-j}(\X^1_{s,s+u}-\X^2_{s,s+u})|^p\lesssim \bigg(\sum_{i=1}^{k-j}\rho^{(i)}_{\hat N^{\delta,p};[s,t]}(\X^1,\X^2)^{\frac{p}{k-j}} \bigg)^{k-j}|t-s|^{(\delta-\frac{1}{p})(k-j)p}.
  \end{align*}
  Combining the last two estimates, we get 
  \begin{align*}
    \Delta_2 \lesssim & \sum_{j=1}^{k-1} \|\X^2\|^{j-1}_{\hat N^{\delta,p};[s,t]} |t-s|^{(\delta -\frac{1}{p})(j-1)} \\
    &\times \sup_{|t-s|\geq h>0} h^{-\delta}\bigg (\int_s^{t-h} |\pi_{k-j}(\X^1_{s,s+u}-\X^2_{s,s+u})|^p \|\X^2_{s+u,s+u+h}\|^{p} \dd u \bigg)^{\frac{1}{p}}\\
    \lesssim & \sum_{j=1}^{k-1} \|\X^2\|^{j}_{\hat N^{\delta,p};[s,t]} |t-s|^{(\delta -\frac{1}{p})(k-1)} \bigg(\sum_{i=1}^{k-j}\rho^{(i)}_{\hat N^{\delta,p};[s,t]}(\X^1,\X^2)^{\frac{p}{k-j}} \bigg)^{\frac{k-j}{p}}. \\
  \end{align*}
  For $\Delta_3$ we briefly need to introduce the inhomogeneous H\"older distance for level $k$ by
  \begin{equation*}
    \rho^{(k)}_{(\delta-1/p)\hol;[s,t]}(\X^1,\X^2):=\sup_{u,v\in[s,t];\,u\neq v} \frac{|\pi_k(\X^1_{s+u,s+v}-\X^2_{s+u,s+v})|}{|u-v|^{(\delta -\frac{1}{p})k}}.
  \end{equation*}
  This time we simply estimate  
  \begin{align*}
    \Delta_3 \leq \rho^{(k)}_{\hat N^{\delta,p};[s,t]}(\X^1,\X^2)^{\frac{1}{k}} \rho^{(k)}_{(\delta-1/p)\hol;[s,t]}(\X^1,\X^2)^{1-\frac{1}{k}}|t-s|^{(\delta -\frac{1}{p})(k-1)}.
  \end{align*}
  Applying Proposition~\ref{prop:variation embeddings} to $Z^s_{\cdot}$ we get
  \begin{equation*}
    |\pi_k(\X^1_{s,t}-\X^2_{s,t})|= |Z^s_0-Z^s_{t-s}| \lesssim \sup_{|t-s|\geq h>0} h^{-\delta}\bigg (\int_s^{t-h} |Z^s_{u+h}-Z^s_{u}|^p \dd u \bigg)^{\frac{1}{p}}|t-s|^{\delta -\frac{1}{p}}.
  \end{equation*}
  Putting the estimates for $\Delta_1$, $\Delta_2$ and $\Delta_3$, we deduce further 
  \begin{align*}
    |\pi_k(\X^1_{s,t}-\X^2_{s,t})| \leq &
    \tilde C |t-s|^{\delta -\frac{1}{p}} \bigg [ \sum_{j=1}^{k-1} \bigg(\sum_{i=1}^{j} \rho^{(i)}_{\hat N^{\delta,p};[s,t]}(\X^1,\X^2)^{\frac{p}{j}}\bigg)^{\frac{j}{p}} |t-s|^{(\delta -\frac{1}{p})(k-1)}\\
    &+\rho^{(k)}_{\hat N^{\delta,p};[s,t]}(\X^1,\X^2)^{\frac{1}{k}}\rho^{(k)}_{(\delta -1/p)\hol;[s,t]}(\X^1,\X^2)^{1-\frac{1}{k}}|t-s|^{(\delta -\frac{1}{p})(k-1)}\bigg ]\allowdisplaybreaks\\
    \lesssim & \tilde C |t-s|^{(\delta -\frac{1}{p})k}  \bigg [\sum_{j=1}^{k} \rho^{(j)}_{\hat N^{\delta,p};[s,t]}(\X^1,\X^2) \\
    &+\bigg(\sum_{j=1}^{k} \rho^{(j)}_{\hat N^{\delta,p};[s,t]}(\X^1,\X^2)^{\frac{p}{k}}\bigg )^{\frac{1}{p}}  \rho^{(k)}_{(\delta -1/p)\hol;[s,t]}(\X^1,\X^2)^{1-\frac{1}{k}}\bigg ]\\
    \lesssim & \tilde C |t-s|^{(\delta -\frac{1}{p})k} \bigg[ \bigg(\sum_{j=1}^{k} \rho^{(j)}_{\hat N^{\delta,p};[s,t]}(\X^1,\X^2)^{\frac{p}{k}}\bigg )^{\frac{k}{p}}  \\
    &+\bigg(\sum_{j=1}^{k} \rho^{(j)}_{\hat N^{\delta,p};[s,t]}(\X^1,\X^2)^{\frac{p}{k}}\bigg )^{\frac{1}{p}}  \rho^{(k)}_{(\delta -1/p)\hol;[s,t]}(\X^1,\X^2)^{1-\frac{1}{k}}\bigg ],
  \end{align*}
  for some constant $\tilde C = \tilde C(\delta,p,\|\X^1\|_{\hat N^{\delta,p}},\|\X^2\|_{\hat N^{\delta,p}})\geq 1$, which can be rewritten as
  \begin{align*}
    \frac{|\pi_k(\X^1_{s,t}-\X^2_{s,t})|}{|t-s|^{(\delta -\frac{1}{p})k}}\lesssim &\tilde C \bigg[ \bigg(\sum_{j=1}^{k} \rho^{(j)}_{\hat N^{\delta,p};[s,t]}(\X^1,\X^2)^{\frac{p}{k}}\bigg )^{\frac{k}{p}} \\
    &+\bigg(\sum_{j=1}^{k} \rho^{(j)}_{\hat N^{\delta,p};[s,t]}(\X^1,\X^2)^{\frac{p}{k}}\bigg )^{\frac{1}{p}} \rho^{(k)}_{(\delta -1/p)\hol;[s,t]}(\X^1,\X^2)^{1-\frac{1}{k}}\bigg ].
  \end{align*}
  In other words, we showed that 
  \begin{align*}
    \frac{\rho^{(k)}_{(\delta -1/p)\hol;[s,t]}(\X^1,\X^2)}{\tilde \omega^{(k)}(s,t)} \lesssim &\tilde C \bigg [1 +\bigg( \frac{\rho^{(k)}_{(\delta -1/p)\hol;[s,t]}(\X^1,\X^2)}{\tilde \omega^{(k)}(s,t)}\bigg)^{1-\frac{1}{k}}\bigg ],
  \end{align*}
  with 
  \begin{equation*}
    \tilde \omega^{(k)}(s,t):= \bigg(\sum_{j=1}^{k} \rho^{(j)}_{\hat N^{\delta,p};[s,t]}(\X^1,\X^2)^{\frac{p}{k}}\bigg )^{\frac{k}{p}}.
  \end{equation*}
  Hence, there exists a constant $C =  C(\delta,p,\|\X^1\|_{\hat N^{\delta,p}},\|\X^2\|_{\hat N^{\delta,p}})\geq 1$ such that 
  \begin{equation*}
    \frac{\rho^{(k)}_{(\delta -1/p)\hol;[s,t]}(\X^1,\X^2)}{\tilde \omega^{(k)}(s,t)} \leq C.
  \end{equation*}
  In particular, we have 
  \begin{equation*}
    |\pi_k(\X^1_{s,t}-\X^2_{s,t})| \leq C \bigg(\sum_{j=1}^{k} \rho^{(j)}_{\hat N^{\delta,p};[s,t]}(\X^1,\X^2)^{\frac{p}{k}}\bigg )^{\frac{k}{p}} |t-s|^{(\delta -\frac{1}{p})k},
  \end{equation*}
  which implies \eqref{eq:control each level} for level $k$ and the proof is complete.
\end{proof}

Combining the local equivalence of the inhomogeneous distances $\rho_{\tilde V^{\delta, p}}$ and $\rho_{\hat N{\delta, p}}$ (Lemma~\ref{lem:nikolskii distance 1} and~\ref{lem:nikolskii distance 2}) with the local Lipschitz continuity of the It\^o-Lyons map~\eqref{eq:ito lyons map} with respect to $\rho_{\tilde V^{\delta, p}}$ (Theorem~\ref{thm:ito map riesz variation}), we deduce same continuity result with respect to $\rho_{\hat N^{\delta, p}}$:

\begin{corollary}\label{cor:ito map nikolskii type norm}
  Let $\delta \in (0,1)$ and $\gamma, p \in (1,\infty)$ be such that $\delta > 1/p$ and $\gamma >1/\delta$. 
  
  The It\^o-Lyons map  
  \begin{equation*}
    \Phi \colon  \mathbb{R}^m \times \Lip^{\gamma} \times \hat N^{\delta,p}([0,T];G^{\lfloor 1/\delta \rfloor}(\mathbb{R}^n)) \to \hat N^{\delta,p}([0,T];\mathbb{R}^m) \quad \text{via} \quad \Phi(y_0,V,\X):=Y,
  \end{equation*}
  where $Y$ denotes the solution to controlled differential equation~\eqref{eq:rough differential equation} given the input $(y_0,V,\X)$, is locally Lipschitz continuous with respect to inhomogeneous Nikolskii type distance $\rho_{\hat N^{\delta,p}}$.
\end{corollary}

\bibliography{quellenRDE}{}
\bibliographystyle{amsalpha}

\end{document}